\documentclass[10pt, leqno]{article}

\usepackage{amsfonts}
\usepackage{amsmath}
\usepackage{amssymb}
\usepackage{amsthm}
\usepackage{mathrsfs}
\usepackage{latexsym}
\usepackage{graphicx}
\usepackage{bm}
\usepackage{inputenc}
\usepackage{enumitem}
\setlist[enumerate,1]{label=(\alph*), ref=(\alph*)}

\usepackage{hyperref}
\usepackage{graphicx}
\usepackage{xcolor}
\usepackage{url}
\usepackage{hyperref}

\usepackage{mathtools}
\mathtoolsset{showonlyrefs}

\usepackage[a4paper,bindingoffset=0.2in,%
left=1.0in,right=1.0in,top=1.0in,bottom=1.0in,%
footskip=.35in]{geometry}

\swapnumbers

\newtheorem{Lemma}{Lemma}[section]
\newtheorem{Corollary}[Lemma]{Corollary}
\newtheorem{Theorem}[Lemma]{Theorem}
\newtheorem*{Theorem*}{Theorem}
\newtheorem*{Corollary*}{Corollary}
\newtheorem*{Conjecture*}{Conjecture}

\theoremstyle{definition}

\newtheorem{Definition}[Lemma]{Definition}

\theoremstyle{remark}
\newtheorem{Remark}[Lemma]{Remark}

\DeclareMathOperator{\lin}{span}
\DeclareMathOperator{\divergence}{div}

\DeclareMathOperator{\without}{\sim}

\DeclareMathOperator{\Clos}{Clos}  
\DeclareMathOperator{\Tan}{Tan}     
\DeclareMathOperator{\spt}{spt}     
\DeclareMathOperator{\im}{im}       
     
\DeclareMathOperator{\grad}{grad}  
\DeclareMathOperator{\dmn}{dmn}     
\DeclareMathOperator{\Nor}{Nor}
\DeclareMathOperator{\Hom}{Hom}     
\DeclareMathOperator{\pt}{pt}       
\DeclareMathOperator{\ap}{ap}  
\DeclareMathOperator{\distance}{distance}

\DeclareMathOperator{\reach}{reach}

\DeclareMathOperator{\trace}{tr}

\DeclareMathOperator{\Int}{Int}

\DeclareMathOperator{\sing}{sing}

\newcommand{\restrict}{ \mathop{ \rule[1pt]{.5pt}{6pt} \rule[1pt]{4pt}{0.5pt} }\nolimits }

\newcommand{\Der}{\ensuremath{\mathrm{D}}}

\newcommand{\ud}{\ensuremath{\,\mathrm{d}}}

\newcommand{\dist}[2]{\distance\left(#1,#2\right)}
\newcommand{\Real}[1]{ \mathbf{R}^{#1}}
\newcommand{\Haus}[1]{ \mathscr{H}^{#1} }
\newcommand{\Leb}[1]{ \mathscr{L}^{#1} }
\newcommand{\rect}[1]{(\mathscr{H}^{#1},#1)}

\newcommand{\Hdensity}[3]{\bm{\Theta}^{#1}(\mathscr{H}^{#1}\restrict \,#2,#3 )}

\newcommand{\nat}{\mathbb{N}}
\newcommand{\grass}[2]{\mathbf{G}(#1,#2)}

\newcommand{\Var}[1]{\mathbf{V}_{#1}}
\newcommand{\RVar}[1]{\mathbf{RV}_{#1}}
\newcommand{\var}[1]{\mathbf{v}_{#1}}
\newcommand{\tbwedge}{{\textstyle \bigwedge}}
\newcommand{\cnt}[1]{\mathscr{C}^{#1}}

\newcommand{\dspace}[2]{\mathscr{D}(#1,#2)}

\newcommand{\VF}{\mathscr{X}}

\newcommand{\project}[1]{{#1}_\natural}

\newcommand{\oball}[2]{\mathbf{U}(#1,#2)}
\newcommand{\oballF}[3]{\mathbf{U}^{#1}(#2,#3)}

\newcommand{\cball}[2]{\mathbf{B}(#1,#2)}
\newcommand{\cballF}[3]{\mathbf{B}^{#1}(#2,#3)}

\newcommand{\sphere}[1]{\mathbf{S}^{#1}}

\newcommand{\id}[1]{\mathrm{id}_{#1}}

\newcommand{\CF}[1]{\bm{1}_{#1}}

\newcommand{\orthproj}[2]{\mathbf{O}^\ast({#1},{#2})}

\newcommand{\density}{\boldsymbol{\Theta}}

\title{Uniqueness of critical points of the anisotropic isoperimetric problem for finite perimeter sets}
\author{Antonio De Rosa \and S{\l}awomir Kolasi{\'n}ski \and Mario Santilli}

\begin{document}

\maketitle
\begin{abstract}
Given an elliptic integrand of class $ \mathscr{C}^{3} $, we prove that finite unions of disjoint open Wulff shapes with equal radii are the only volume-constrained critical points of the anisotropic surface energy among all sets with finite perimeter and reduced boundary almost equal to its closure.
\end{abstract}

\section{Introduction}

The classical anisotropic isoperimetric problem (or Wulff problem) consists in minimizing the anisotropic boundary energy among all sets of finite perimeter with prescribed volume. For all positive (continuous) integrands the solution is uniquely characterized, up to translation, by the Wulff shape, as proved by Taylor in \cite{MR0388225}. Alternative proofs can be found in \cite{fonsecamuellerwulff, Gromov,
  brothersmorgan}. This isoperimetric shape was constructed by Wulff  in \cite{Wulff}
  and plays a central role in crystallography. 

Instead of considering minima, a more subtle question is to characterize \emph{critical points} of the anisotropic isoperimetric problem. For integrands of class $ \mathscr{C}^{1} $, this is equivalent to characterize sets of finite perimeter whose anisotropic mean curvature in the sense of varifolds is constant. For all convex integrands in $ \Real{2} $, Morgan proved in \cite{MR2113931} that  Wulff shapes are the only critical points among all planar regions with boundary given by a closed and connected rectifiable curve. To the best of our knowledge, the characterization in every dimension for smooth boundaries has been conjectured for the first time by Giga in \cite{Gigacon} and Morgan in \cite{MR2113931}. This has been positively answered for smooth elliptic integrands in \cite{Giga} for dimension $3$, and in \cite{helimage} for every dimension. These works are the anisotropic counterpart of the celebrated Alexandrov's result \cite{MR0102114}.  Moreover, quantitative stability versions of this rigidity theorem have been showed in \cite{DMMN,DG,DG2}.

In the non-smooth setting, Maggi has conjectured in \cite[Conjecture]{OR} the characterization of the Wulff shapes among \emph{sets of finite perimeter}:
\begin{Conjecture*}[\cite{OR}]
 F-Wulff shapes are the unique sets of finite perimeter and finite volume that are critical points of $ \mathcal F$ at fixed volume.
\end{Conjecture*}
Since $F$ is assumed to be convex, but may fail to be $ \mathscr{C}^{1} $, the notion of first variation and critical points are suitably defined in \cite[p. 35-36]{OR}, using the convexity in time of the functional along any prescribed variational flow. Maggi specifies in \cite{OR} the significant interest from the physical viewpoint for crystalline integrands. Moreover he points out that this question is open even for smooth elliptic anisotropic energies and among sets with Lipschitz boundary. 

Delgadino and Maggi have settled the conjecture \cite[Conjecture]{OR} for the special case of the \emph{area functional} in \cite{DM00pub}, proving
that among sets of finite perimeter, finite unions of balls
with equal radii are the unique volume-constrained critical points of the isotropic surface area. Their beautiful proof provides a measure-theoretic revisiting of the Montiel-Ros argument \cite{montielros} by means of the Heintze-Karcher inequality. Recently the third author has obtained with different techniques in \cite{2019arXiv190805952S} a similar Heintze-Karcher inequality for sets of finite perimeter and bounded isotropic mean curvature and has proved that the equality case is uniquely characterized by finite unions of disjoint open balls, thus recovering the characterization of isotropic critical points. 

To deal with the lack of regularity of finite perimeter sets, Delgadino and Maggi need to use in \cite{DM00pub} the strong maximum principle for integral varifolds of Sch\"atzle \cite{schatzle}. Unfortunately, as they point out, this result is only available in the isotropic setting, preventing the extension of the method of \cite{DM00pub} to anisotropic integrands. They can threat in \cite{DM0} the special case of \emph{local minimizers}, since this allows to apply an anisotropic strong maximum principle proved in \cite{DPM} and provides the required regularity through the use of suitable competitors. Nevertheless these competition arguments are not applicable to study the general case of \emph{anisotropic critical points}.

In the present paper we address this problem, providing a positive answer to \cite[Conjecture]{OR} for elliptic integrands of class $ \mathscr{C}^{3} $ among finite perimeter sets with reduced boundary almost equal to its closure, see Corollary \ref{main corollary}.
Our main result, see Theorem~\ref{Heintze-Karcher}, is actually more general and it consists in the following anisotropic Heintze-Karcher inequality for sets of finite perimeter (we~refer to Section~\ref{prel} for the notation) and in the characterization of finite unions of disjoint open Wulff shapes (of possibly different radii) as the unique configurations realizing the equality case.
\begin{Theorem*}
  Suppose $F$ is an elliptic integrand of class~$\cnt{3}$ (see \ref{def:elliptic}), $\alpha \in (0,1)$, $c \in (0,\infty)$, $E \subseteq \Real{n+1}$ is
    a~set of finite perimeter such that $\Haus{n}(\Clos(\partial^{\ast} E)
    \without \partial^{\ast}E) = 0$ and the distributional anisotropic mean
    curvature~$H$ of $\partial^* E$ with respect to $ F $ in the direction of the interior normal satisfies $0 < H \le c$ and it is locally of class~$\cnt{0,\alpha}$
    on the~$\cnt{1,\alpha}$~regular part of $\spt \|V\|$. Then
    \begin{equation*}
        \Leb{n+1}(E) \leq \frac{n}{n+1}\int_{\partial E}\frac{F(\mathbf{n}(E,x))}{H(x)} \ud \Haus{n}(x) \,.
    \end{equation*}
    Equality holds if and only if $E$ coincides up to a set of
    $\Leb{n+1}$~measure zero with a finite union of disjoint open Wulff shapes
    with radii not smaller than $n/c$.
\end{Theorem*}

As mentioned before, for any elliptic integrand of class $ \mathscr{C}^{3} $, we obtain the following characterization of finite unions of Wulff shapes as the only volume-constrained anisotropic critical points among finite perimeter sets with reduced boundary almost equal to its closure, see Corollary \ref{main corollary}. We denote by $ \mathcal{P}_{F} $ \emph{the $ F $-perimeter functional}, i.e.\ 
\begin{equation*}
 \mathcal{P}_{F}(E) = \int_{\partial^{\ast}E} F(\mathbf{n}(E,x))\, \ud \Haus{n}(x)
\end{equation*}
for every $ E \subseteq \Real{n+1} $ with finite perimeter.

\begin{Corollary*}
   Suppose $ E \subseteq \Real{n+1} $ is a finite perimeter set with finite volume such that
   \begin{equation*}
   	\Haus{n}(\Clos(\partial^{\ast} E) \without \partial^{\ast}E) = 0.
   \end{equation*}
  
  If $ E $ is a volume-constrained critical point of $ \mathcal{P}_{F} $, then $ E $ is equivalent to a finite union of disjoint open Wulff shapes.
\end{Corollary*}

We describe now the structure of the paper. In Section \ref{prel}, after having recalled some background material, we provide some classical facts on Wulff shapes and we study some basic properties of the
anisotropic nearest point projection onto an arbitrary closed set. In Section \ref{tota} we prove that the
only totally umbilical closed and connected hypersurface of class~$\cnt{1,1}$ is
the Wulff shape. In Section \ref{lusi} we recall the notion of anisotropic
$(n,h)$-sets introduced in \cite{2019arXiv190103514D} and we prove that their generalized normal bundle satisfies a Lusin~(N) condition with respect to the $ n $ dimensional Hausdorff measure $ \Haus{n} $, thus extending an analogous result for isotropic $(n,h)$ sets obtained in \cite[3.7]{2019arXiv190310379S}. This is the key to obtain the main result of the paper. In~Section \ref{anis} we introduce the anisotropic normal bundle and we study its relation with the isotropic one and with the anisotropic nearest point projection; moreover we consider the anisotropic Steiner formula for closed
sets and we prove that every closed set satisfying such a formula has
positive reach. To conclude, in Section \ref{secmain} we combine all these tools to prove Theorem~\ref{Heintze-Karcher} and Corollary \ref{main corollary}.

\section{Preliminaries}\label{prel}

\subsection*{Notation}
The natural number $n \ge 1$ shall be fixed for the whole paper.

In~principle, but with some exceptions explained below, we shall follow the
notation of Federer (see~\cite[pp. 669 -- 671]{Federer1969}). Whenever $A
\subseteq \Real{n+1}$ we denote by $\Clos{A}$ the closure of~$A$
in~$\Real{n+1}$. Following Almgren (e.g.~\cite{Almgren2000}) if $T \in
\grass{n+1}{k}$, then we~write $\project{T}$ for the linear orthogonal
projection of~$\Real{n+1}$ onto~$T$. The symbol $\nat$ stands for the set of
non-negative integers. We~use standard abbreviations for intervals $(a,b) =
\Real{} \cap \{ t : a < t < b \}$ and $[a,b] = \Real{} \cap \{ t : a \le t \le b
\}$. We also employ the terminology introduced in \cite[3.2.14]{Federer1969}
when dealing with rectifiable sets. Moreover, given a measure $\phi$ and
a~positive integer~$m$ the notions of $(\phi,m)$~approximate tangent cone
$\Tan^{m}(\phi,\cdot)$, $(\phi,m)$~approximate differentiability and
$(\phi,m)$~approximate differential are used in agreement with
\cite[3.2.16]{Federer1969}. We also introduce the symbol $\sphere{n}$ for the
unit $n$-dimensional sphere in~$\Real{n+1}$.

Concerning varifolds and submanifolds of~$\Real{n+1}$ we use the notation
introduced in~\cite{Allard1972}. If $M$ is a~submanifold of~$\Real{n+1}$ of
class~$\cnt{1}$, we write $\VF(M)$ for compactly supported tangent vectorfields
on~$M$ of class~$\cnt{1}$; cf.~\cite[2.5]{Allard1972}. We say that $M$ is a
\emph{closed submanifold} of~$\Real{n+1}$ if it is a~submanifold of~$\Real{n+1}$
and a~closed (but not necessarily compact) subset of~$\Real{n+1}$;
in~particular, $\partial M \without M = \varnothing$.

We also use the following convention.  Whenever $X$, $Y$ are vectorspaces, $A
\subseteq X$, and $f : A \to Y$ we write $\Der f$ for the derivative of~$f$ that
is a~$\Hom(X,Y)$~valued function whose domain is the set of points
of~differentiability of~$f$. If~$Y = \Real{}$ and $X$ is equipped with a~scalar
product, then we write $\grad f$ for the $X$~valued function characterised by
\begin{displaymath}
    \bigl\langle u ,\, \Der f(x) \bigr\rangle = \grad f(x) \bullet u
    \quad \text{for $x \in \dmn \Der f$ and $u \in X$} \,.
\end{displaymath}

\subsection*{Pointwise differentiability}

\begin{Definition}[\protect{cf.~\cite[\S{2.7}]{snulmenn:sets}}]
    Let $k \in \nat$, $X$, $Y$ be normed vectorspaces, $A \subseteq X$, $f : A
    \to Y$, and $a \in X$. Then $f$ is called \emph{pointwise differentiable of
      order~$k$ at~$a$} if there exists an open set $U \subseteq X$ and
    a~function $g : U \to Y$ of class~$k$ such that
    \begin{displaymath}
        a \in U \subseteq A \,,
        \quad
        f(a) = g(a) \,,
        \quad \text{and} \quad
        \lim_{x \to a} \frac{|f(x) - g(x)|}{|x-a|^k} = 0 \,.
    \end{displaymath}
    Whenever this is satisfied one defines also the \emph{pointwise differential
      of order~$i$ of~$f$ at~$a$} by
    \begin{displaymath}
        \pt \Der^i f(a) = \Der^i g(a)
        \quad \text{for $i \in \{ 0, 1, \ldots, k \}$} \,.
    \end{displaymath}
\end{Definition}

\begin{Definition}[\protect{cf.~\cite[\S{3.3}]{snulmenn:sets}}]
    Suppose $k,n \in \nat$ and $A \subseteq \Real{n+1}$. Then $A$ is called
    \emph{pointwise differentiable of order~$k$ at~$a$} if there exists
    a~submanifold~$B$ of~$\Real{n+1}$ of class~$k$ such that $a \in B$,
    \begin{gather}
        \lim_{r \downarrow 0} r^{-1} \sup | \dist{\cdot}{A} - \dist{\cdot}{B} | [\cball ar] = 0 \,,
        \\
        \text{and} \quad
        \lim_{r \downarrow 0} r^{-k} \sup \dist{\cdot}{B} [A \cap \cball ar] = 0 \,.
    \end{gather}
\end{Definition}

\begin{Definition}[\protect{cf.~\cite[\S{3.12}]{snulmenn:sets}}]
    Suppose $n,k \in \nat$ and $A \subseteq \Real{n+1}$. Then $\pt \Der^k A$ is
    the function whose domain consists of pairs $(a,S)$ such that $a \in
    \Clos{A}$, $A$ is pointwise differentiable of order~$k$ at~$a$, $S \in
    \grass{n+1}{\dim \Tan(A,a)}$, and $S^{\perp} \cap \Tan(A,a) = \{0\}$ and whose
    value at $(a,S)$ equals the unique $\phi \in \bigodot^k(\Real{n+1},\Real{n+1})$
    such that whenever $f : S \to S^{\perp}$ is of class~$k$ and satisfies
    \begin{gather}
        \lim_{r \downarrow 0} r^{-1} \sup | \dist{\cdot}{A} - \dist{\cdot}{B} | [\cball ar] = 0 \,,
        \\
        \text{and} \quad
        \lim_{r \downarrow 0} r^{-k} \sup \dist{\cdot}{B} [A \cap \cball ar] = 0 \,,
    \end{gather}
    where $B = \{ x + f(x) : x \in S \}$, then $\phi = \Der^k (f \circ \project
    S)(a)$.
\end{Definition}

\begin{Remark}[\protect{cf.~\cite[\S\S{3.14, 3.15}]{snulmenn:sets}}]
    Assume $n,d,k \in \nat$, $S \in \grass {n+1}d$, $U \subseteq S$ is open, $f : U
    \to S^{\perp}$ is continuous, $x \in U$, $A = \{ \chi + f(\chi) : \chi \in S
    \}$. Then $A$ is pointwise differentiable of order~$k$ at~$a = x + f(x)$ if
    and only if $f$ is pointwise differentiable of order~$k$ at~$x$. Moreover,
    $\pt \Der^i A(a,S) = \pt \Der^i (f \circ \project S)(x)$ for $i \in \{ 0, 1,
    \ldots, k \}$.
\end{Remark}

\subsection*{The unit normal bundle of a closed set}
Let $ A \subseteq \Real{n+1} $ be a closed set. 

\begin{Definition}
Given $ A \subseteq \Real{n+1} $ we define \emph{the distance function to $ A $} as 
\begin{equation*}
\bm{\delta}_{A}(x) = \inf\{ |x-a|:a \in A   \} \quad \textrm{for every $ x \in \Real{n+1} $.}
\end{equation*}
Moreover, 
\begin{equation*}
S(A,r) = \{ x : \bm{\delta}_{A}(x) = r \} \quad \textrm{for $ r > 0 $.}
\end{equation*}
\end{Definition}

\begin{Remark}[\protect{cf.\ \cite[2.13]{2017arXiv170801549S}}]
If $ r > 0 $ then $ \Haus{n}(S(A,r)\cap
K)< \infty $ whenever $ K \subseteq \Real{n} $ is compact and $ S(A,r) $ is
countably $ \rect{n} $ rectifiable of class $ 2 $.
\end{Remark}

\begin{Definition}[\protect{cf.\ \cite[3.1]{2017arXiv170801549S}}]
If $U$ is the set of all $x \in \Real{n+1}$ such that there exists a unique $a
\in A$ with $|x-a| = \bm{\delta}_{A}(x)$, we define the \textit{nearest
	point projection onto~$A$} as the map $\bm{\xi}_{A}$ characterised by the
requirement
\begin{displaymath}
| x- \bm{\xi}_{A}(x)| = \bm{\delta}_{A}(x) \quad \textrm{for $x \in U$}.
\end{displaymath}
We set $U(A) = \dmn \bm{\xi}_{A} \without A$. The functions $ \bm{\nu}_{A} $ and $
\bm{\psi}_{A} $ are defined by
\begin{displaymath}
\bm{\nu}_{A}(z) = \bm{\delta}_{A}(z)^{-1}(z - \bm{\xi}_{A}(z))
\quad \textrm{and} \quad
\bm{\psi}_{A}(z)= (\bm{\xi}_{A}(z), \bm{\nu}_{A}(z)),
\end{displaymath}
whenever $ z \in U(A)$.
\end{Definition}

\begin{Definition}[\protect{cf.\ \cite[3.6, 3.8, 3.13]{2017arXiv170801549S}}]\label{regular ponints}
We define the function $\rho(A, \cdot)$ setting
\begin{displaymath}
\rho(A,x) = \sup \bigl\{
t : \bm{\delta}_{A}(\bm{\xi}_{A}(x) + t  (x-\bm{\xi}_{A}(x) )) = t \bm{\delta}_{A}(x)
\bigr\}
\quad \textrm{for $ x \in U(A)$} \,,
\end{displaymath}
and we say that $ x \in U(A) $ is a \emph{regular point of $ \bm{\xi}_{A} $}
if and only if $ \bm{\xi}_{A}$ is approximately differentiable at $ x $ with
symmetric approximate differential and $ \ap \lim_{y \to x} \rho(A,y)
\geq \rho(A,x)>1$. The set of regular points of
$\bm{\xi}_{A}$ is denoted by $R(A)$.

For $\tau \ge 1$ we define
\begin{displaymath}
A_{\tau} = U(A) \cap \{ x : \rho(A,x) \ge \tau \} \,.
\end{displaymath}
\end{Definition}

\begin{Remark}[\protect{cf.\ \cite[3.7]{2017arXiv170801549S}}]
The function $ \rho(A, \cdot) $ is upper semicontinuous and its image is contained $ [1, \infty] $.
\end{Remark}

\begin{Definition}[\protect{cf.\ \cite[4.9]{2017arXiv170801549S}}]
Suppose $ x \in R(A) $. Then $\chi_{A,1}(x) \leq \ldots \leq \chi_{A,n}(x)$ denote the eigenvalues of the symmetric linear map $ \ap \Der \bm{\nu}_{A}(x)| \{v : v \bullet \bm{\nu}_{A}(x) =0  \} $.
\end{Definition}

\begin{Remark}
 Notice that $ \Haus{n}(S(A,r) \sim R(A)) =0 $ for $ \Leb{1} $ a.e.\ $ r > 0 $ (cf.\ \cite[3.16]{2017arXiv170801549S}) and 
 \begin{equation*}
 \Tan^{n}(\Haus{n}\restrict S(A,r), x) =  \{v : v \bullet \bm{\nu}_{A}(x) =0  \}
 \end{equation*}
 for $ \Haus{n} $ a.e.\ $ x \in S(A,r) $ and for $ \Leb{1} $ a.e.\ $ r > 0 $, cf.\ \cite[3.12]{2017arXiv170801549S}. 
 
 The functions $ \chi_{A,i} $ are the approximate principal curvatures of $ S(A,r) $ in the direction of $ \bm{\nu}_{A}(x) $. In fact, as proved in \cite[3.12]{2017arXiv170801549S}, they coincide with the eigenvalues the approximate second-order differential $ \ap \Der^{2}S(A,r) $ of $ S(A,r) $; cf.\ \cite{San} for the general theory of higher order approximate differentiability for sets.
\end{Remark}

\begin{Definition}[\protect{cf.\ \cite[4.1]{2017arXiv170801549S}, \cite[\S 2.1]{MR2031455}}]
 \emph{The generalized unit normal bundle of $ A $} is defined as
\begin{displaymath}
N(A) = (A \times \sphere{n}) \cap \{ (a,u) :
\bm{\delta}_{A}(a+su)=s \; \textrm{for some $ s > 0 $}\}
\end{displaymath} and $ N(A,a) = \{ v : (a,v) \in N(A) \} $ for $ a \in A $.
\end{Definition}

\begin{Remark}[\protect{cf.\ \cite[4.3]{2017arXiv170801549S}}]
	The set $ N(A) $ is a countably $ n $ rectifiable subsets of $ \Real{n+1} \times \mathbf{S}^{n} $.
\end{Remark}

\subsection*{Anisotropic integrands and mean curvature}

\begin{Definition}\label{integrands}
    Let $k \in \nat$, $\alpha \in [0,1]$. By an \emph{integrand of
      class~$\cnt{k,\alpha}$} we mean a~non-negative function $F :\Real{n+1} \to
    \Real{}$ such that $F|{\Real{n+1} \without \{0\}}$ is of
    class~$\cnt{k,\alpha}$ and
    \begin{equation}
        \label{eq:F-one-homog}
        F(\lambda \nu) = |\lambda| F(\nu)
        \quad \text{for $\nu \in \Real{n+1}$ and $\lambda \in \Real{}$}
    \end{equation}

    By an \emph{integrand} we mean an integrand of class~$\cnt{0}$.
\end{Definition}

\begin{Remark}\label{integrands: remark}
    \label{rem:convex-integrands}
    If $F$ is convex, then it is a~norm on~$\Real{n+1}$. We say that $F$ is
    a~\emph{strictly convex norm} if it is an integrand satisfying
    \begin{displaymath}
        F(x+y) < F(x) + F(y)
        \quad \text{for all linearly independent $x,y \in \Real{n+1}$} \,.
    \end{displaymath}
\end{Remark}

\begin{Definition}[\protect{cf.~\cite[5.1.2]{Federer1969} and~\cite[3.1(4)]{AllardReg}}]
    \label{def:elliptic}
    We say that an integrand~$F$ is~\emph{elliptic} if there exists a number
    $\gamma > 0$ such that the map $\Real{n+1} \ni u \mapsto F(u) - \gamma |u|$
    is convex. We call $\gamma$ the \emph{ellipticity constant of~$F$}.
\end{Definition}

\begin{Remark}[\protect{cf.~\cite[5.1.3]{Federer1969}}]
    Assume $F$ is an integrand of class~$\cnt{1,1}$. Then ellipticity of~$F$
    with ellipticity constant $\gamma > 0$ is equivalent to the condition
    \begin{equation}
        \label{ell2}
        \bigl\langle (v,v) ,\, \Der^2 F(u) \bigr\rangle
        \ge \gamma \frac{|u \wedge v|^2}{|u|^3}
        = \gamma \frac{|v|^2 - (v \bullet u/|u|)^2}{|u|}
        \quad \text{for $u \in \dmn \Der^2 F$, $u \ne 0$, $v \in \Real{n+1}$} \,.
    \end{equation}
    In~particular, if $F$ is elliptic, $u \in \dmn \Der^2 F$, $|u| = 1$, and $v
    \in \lin\{ u \}^{\perp}$, then
    \begin{displaymath}
        \bigl\langle (v,v) ,\, \Der^2 F(u) \bigr\rangle
        \ge \gamma |v|^2 \,,
    \end{displaymath}
    which shows that $F$ is uniformly elliptic in the sense
    of~\cite[\S{2}]{2019arXiv190103514D}.
\end{Remark}

\begin{Definition}
    \label{def:CF}
    Assume $F$ is an elliptic integrand with ellipticity constant $\gamma >
    0$. We define
    \begin{displaymath}
        C(F) = \sup \bigl(
        \bigl\{ \gamma^{-1} ,\, \sup F[\sphere{n}] / \inf F[\sphere{n}] \bigr\}
        \cup \bigl\{ \| \Der^2 F(\nu) \| : \nu \in \sphere{n} \cap \dmn \Der^2 F \bigr\} 
        \bigr) \,.
    \end{displaymath}
\end{Definition}

\begin{Remark}
    Let $U \subseteq \Real{n+1}$ be open. For any $T \in \grass{n+1}{n}$ we
    choose arbitrarily $\nu(T) \in T^{\perp}$ such that $|\nu(T)| = 1$. In the
    sequel we shall tacitly identify any $V \in \Var{n}(U)$ with a~Radon
    measure~$\bar V$ over~$U \times \Real{n+1}$ such that
    \begin{displaymath}
        \bar V(\alpha) = \frac 12 \int \alpha(x,\nu(T)) + \alpha(x,-\nu(T)) \ud V(x,T)
        \quad \text{for $\alpha \in C^0_c(U,\Real{})$} \,.
    \end{displaymath}
    Clearly, this definition does not depend on the choice of~$\nu(T)$.
\end{Remark}

\begin{Definition}
    \label{def:fst-var-wrt-F}
    Let $U \subseteq \Real{n+1}$ be open, $F$ be an integrand of
    class~$\cnt{1}$, $V \in \Var{n}(U)$. We~define the \emph{first variation
      of~$V$ with respect to~$F$} by the formula
    \begin{displaymath}
        \delta_F V(g) =  \int \Der g(x) \bullet B_F(\nu) \ud V(x,\nu)
        \quad \text{for $g \in \VF(U)$} \,,
    \end{displaymath}
    where $B_F(\nu) \in \Hom(\Real{n+1}, \Real{n+1})$ is given by
    \begin{displaymath}
        B_F(\nu)u = F(\nu) u - \nu \cdot \langle u ,\, \Der F(\nu) \rangle
        \quad \text{for $\nu,u \in \Real{n+1}$, $\nu \ne 0$} \,.
    \end{displaymath}
\end{Definition}

\begin{Remark}[\protect{cf.~\cite{AllardReg},\cite[Appendix~A]{DPDRG2018},\cite{DeRosa}}]
    \label{rem:fst-var}
    If $\varphi : \Real{} \times \Real{n+1} \to \Real{n+1}$ is smooth, $\varphi(0,x)
    = x$ for $x \in \Real{n+1}$, and $g = \frac{\ud}{\ud t}|_{t=0}
    \varphi(t,\cdot) \in \VF(\Real{n+1})$, then
    \begin{displaymath}
        \left. \frac{\ud}{\ud t} \right|_{t=0} \Phi_F(\varphi_{t\#}V) = \delta_FV(g) \, ,
    \end{displaymath}
    where the functional $\Phi_F :
    \Var{n}(U) \to [0,\infty]$ is defined as 
    \begin{displaymath}
        \Phi_F(V) = \int F(\nu) \ud V(x,\nu) \,.
    \end{displaymath}
\end{Remark}

\begin{Definition}[\protect{cf.~\cite[\S{2}]{2019arXiv190103514D}}]
    \label{def:F-mean-curvature}
    Let $\Omega \subseteq \Real{n+1}$ be open, $V \in \Var{n}(\Omega)$, $F :
    \Real{n+1} \to \Real{}$ be an integrand of class~$\cnt{1}$. Assume that
    $\| \delta_F V \|$ is a~Radon measure. Then
    \begin{displaymath}
        \delta_FV(g)
        = - \int \overline{\mathbf{h}}_F(V,x) \bullet g(x) \ud \|V\|(x)
        + \int \boldsymbol{\eta}_F(V,x) \bullet g(x) \ud \|\delta_F V\|_{\mathrm{sing}}(x)
        \quad \text{for $g \in \mathscr{X}(\Omega)$} \,,
    \end{displaymath}
    where $\|\delta_F V\|_{\mathrm{sing}}$ is the singular part of~$\|\delta_F
    V\|$ with respect to~$\|V\|$, $\overline{\mathbf{h}}_F(V,\cdot)$ is
    an~$\Real{n+1}$~valued $\|V\|$-integrable function, and
    $\boldsymbol{\eta}_F(V,\cdot)$ is an~$\sphere{n}$~valued $\| \delta_F V
    \|$-integrable function.

    For $\|V\|$-a.e.~$x$ we define the~\emph{$F$-mean curvature vector
      of~$V$ at~$x$}, denoted $\mathbf{h}_F(V,x)$, by the formula
    \begin{displaymath}
        \mathbf{h}_F(V,x) = \frac{\overline{\mathbf{h}}_F(V,x)}{\int F(\nu) \ud V^{(x)}(\nu)} \,,
    \end{displaymath}
    where $V^{(x)}$ is the probability measure on~$\sphere{n}$ coming from
    disintegration of~$V$; see~\cite[\S{3.3}]{Allard1972}.
\end{Definition}

\begin{Definition}
    Define $\bm{\Xi} : \bigodot^2 \Real{n+1} \to \Hom(\Real{n+1},\Real{n+1})$ to be
    the~linear map characterised by
    \begin{displaymath}
        \langle u ,\, \bm{\Xi}(A) \rangle \bullet v = A(u,v)
        \quad \text{for $A \in {\textstyle \bigodot^2} \Real{n+1}$ and $u,v \in \Real{n+1}$} \,.
    \end{displaymath}  
\end{Definition}

\begin{Remark}
    In particular, if $f : \Real{n+1} \to \Real{}$ is twice differentiable at $x
    \in \Real{n+1}$, then
    \begin{displaymath}
        \bm{\Xi}(\Der^2f(x)) = \Der(\grad f)(x) \in \Hom(\Real{n+1},\Real{n+1}) \,.
    \end{displaymath}
\end{Remark}

\begin{Remark}
    \label{rem:Fmc-for-smooth}
    Let $G \subseteq \Real{n+1}$ be open, $v_1,\ldots,v_{n+1}$ be an orthonormal
    basis of $\Real{n+1}$, $M$ be a submanifold of~$G$ of dimension~$n$ of
    class~$\cnt{2}$, $V = \var{n}(M) \in \Var{n}(G)$, $x \in M$, $\nu : G \to
    \Real{n+1}$ be of class~$\cnt{1}$ and satisfy
    \begin{equation}
        \label{eq:nu-conditions}
        |\nu(y)| = 1 \,,
        \quad
        \nu(y) \in \Nor(M,y) \,,
        \quad \text{and} \quad
        \langle \nu(y) ,\, \Der \nu(y) \rangle = 0
        \quad \text{for $y \in M$} \,.
    \end{equation}
    In~\cite[Proposition~2.1]{2019arXiv190103514D} the authors show that if $ F $ is an elliptic integrand of class $ \mathscr{C}^{2} $ then 
    \begin{displaymath}
        - F(\nu(x)) \mathbf{h}_F(V,x)
        = \nu(x) \trace \bigl( \Der (\grad F \circ \nu)(x)\bigr)
        = \nu(x) \sum_{j = 1}^{n+1} 
        \bigl\langle (\Der \nu (x) v_j, v_j) ,\, \Der^2 F(\nu(x)) \bigr\rangle
        \,.
    \end{displaymath}
\end{Remark}

\begin{Definition}[\protect{cf.~\cite[4.5.5]{Federer1969}}]
    \label{def:out-normal-vector}
    Let $A \subseteq \Real{n+1}$ and $b \in \Real{n+1}$. We say that $u$ is
    an~\emph{exterior normal of~$A$ at~$b$} if $u \in \Real{n+1}$, $|u| = 1$,
    \begin{gather}
        \density^{n+1}(\Leb{n+1} \restrict \{ x : (x-b) \bullet u > 0 \} \cap A,b) = 0 \,,
        \\
        \text{and} \quad
        \density^{n+1}(\Leb{n+1} \restrict \{ x : (x-b) \bullet u < 0 \} \without A,b) = 0 \,.
    \end{gather}
    We also set $\mathbf{n}(A,b) = u$ if $u$ is the exterior normal of~$A$
    at~$b$ and $\mathbf{n}(A,b) = 0$ if there exists no exterior normal of~$A$
    at~$b$.
\end{Definition}

\begin{Definition}
    \label{def:anis-normal}
    Let $E \subseteq \Real{n+1}$ and $x \in \Real{n+1}$. We define 
    \begin{displaymath}
        \mathbf{n}^{F}(E,x) = \grad F(\mathbf{n}(E,x)) \quad \text{if $\mathbf{n}(E,x) \ne 0$}
        \quad \text{and} \quad
        \mathbf{n}^{F}(E,x) = 0  \quad \text{if $\mathbf{n}(E,x) = 0$} \,.
    \end{displaymath}
\end{Definition}

\begin{Remark}
    \label{rem:two-self-adjoint}
    Assume $X$ is a~Hilbert space, $\dim X = k \in \nat$, $A,B \in \Hom(X,X)$
    are self-adjoint automorphisms of~$X$, and $A$ is positive definite. With
    the help of the (tiny) spectral theorem~\cite[Chap.~VIII,
    Thm.~4.3]{LangLA3rd} we find a self-adjoint and positive definite map $C \in
    \Hom(X,X)$ such that $A = C \circ C$. Next, we observe that $E = C^{-1}
    \circ A \circ B \circ C = C \circ B \circ C$ is self-adjoint. Employing
    again the (tiny) spectral theorem we find an orthonormal basis
    $v_1,\ldots,v_{k} \in X$ and real numbers $\lambda_1,\ldots,\lambda_{k}$
    such that $E v_i = \lambda_i v_i$ for $i \in \{1,2,\ldots,k\}$. We obtain
    \begin{displaymath}
        A \circ B (Cv_i) = C \circ E v_i = \lambda_i Cv_i \quad \text{for $i \in \{1,2,\ldots,k\}$} 
    \end{displaymath}
    and we see that $Cv_1,\ldots,Cv_{k}$ is a basis of eigenvectors of~$A \circ
    B$ with eigenvalues $\lambda_1,\ldots,\lambda_{k}$.

    In particular, if $G$, $M$, $x$, and $\nu$ are as
    in~\ref{rem:Fmc-for-smooth}, $F$ is an elliptic integrand, $u = \nu(x) \in
    \dmn \Der^2 F$, and $X = \Tan(M,x)$, then the maps $A = \bm{\Xi}( \Der^2
    F(\nu(x)) | X \times X )$ and $B = \Der \nu(x)|X \in \Hom(X,X)$ are
    self-adjoint and $A$ is positive definite; hence, $A \circ B$ has exactly
    $n$~real eigenvalues.

    Observe also that since $F$ is positively $1$-homogeneous, $\grad F$ is
    positively $0$-homogeneous, i.e., $\grad F(\lambda v) = \grad F(v)$ for
    $\lambda \in (0,\infty)$ and $v \in \dmn \grad F$; hence,
    \begin{equation}
        \label{eq:v-in-ker-DgradFv}
        v \in \ker \Der (\grad F)(v)
        \quad \text{for $v \in \dmn \Der^2 F$} \,.
    \end{equation}
    Since $\Der^2F(\nu(x)) \in \bigodot^2 \Real{n+1}$ is symmetric it follows
    that $\Der(\grad F)(\nu(x)) \in \Hom(\Real{n+1},\Real{n+1})$ is self-adjoint
    and we have
    \begin{displaymath}
        \im \Der(\grad F)(\nu(x)) = \bigl( \ker \Der(\grad F)(\nu(x)) \bigr)^{\perp}
    \end{displaymath}
    so that $\Der(\grad F)(\nu(x))|X \in \Hom(X,X)$
    by~\eqref{eq:v-in-ker-DgradFv}. Seeing that also $\Der \nu(x)|X \in
    \Hom(X,X)$ we conclude
    \begin{displaymath}
        \Der(\grad F \circ \nu)(x)|X \in \Hom(X, X) \,.
    \end{displaymath}
\end{Remark}

\begin{Definition}
    \label{def:F-princ-curv}
    Let $F$ be an elliptic integrand of class~$\cnt{1,1}$, $G \subseteq
    \Real{n+1}$ be open, $M$ be a~submanifold of~$G$ of dimension~$n$ of
    class~$\cnt{1,1}$, $\nu : G \to \Real{n+1}$ be Lipschitz continuous and such
    that $|\nu(z)| = 1$ and $\nu(z) \in \Nor(M,z)$ for $z \in M$, $x \in \dmn
    \Der \nu$, and $u = \grad F(\nu(x))$. We define the \emph{$F$-principal
      curvatures of~$M$ at~$(x,u)$}
    \begin{displaymath}
        \kappa_{M,1}^F(x,u) \le \ldots \le \kappa_{M,n}^F(x,u) 
    \end{displaymath}
    to be the eigenvalues of the map $\Der(\grad F \circ \nu)(x)|\Tan(M,x) \in
    \Hom(\Tan(M,x), \Tan(M,x))$; cf.~\ref{rem:two-self-adjoint}.
\end{Definition}

\begin{Remark}
    \label{rem:F-princ-curc-and-mc}
    Clearly if $V = \var{n}(M) \in \Var{n}(G)$, then
    \begin{displaymath}
        \overline{\mathbf{h}}_F(V,x) = -\nu(x) {\textstyle \sum_{i=1}^n} \kappa_{M,i}^F(x,u) \,.
    \end{displaymath}
\end{Remark}

\begin{Definition}
    \label{def:ptwise-F-mc}
    Assume $M \subseteq \Real{n+1}$ is pointwise differentiable of order~$2$
    at~$a \in \Clos{M}$, $T \in \grass{n+1}{n}$, $f : T \to T^{\perp}$ is
    pointwise differentiable of order~$2$ at~$0$, $f(0) = 0$, $\pt \Der f(0) =
    0$, $B = \Real{n+1} \cap \{ a + x + f(x) : x \in T \}$, $\nu \in T^{\perp}$,
    $|\nu| = 1$, and
    \begin{gather}
        \lim_{r \downarrow 0} r^{-1} \sup | \dist{\cdot}{M} - \dist{\cdot}{B} | [\cball ar] = 0 \,,
        \\
        \text{and} \quad
        \lim_{r \downarrow 0} r^{-2} \sup \dist{\cdot}{B} [M \cap \cball ar] = 0 \,.
    \end{gather}
    We define the \emph{pointwise $F$-mean curvature vector of~$M$ at~$a$}, denoted
    $\pt \mathbf{h}_F(M,a)$, by the formula
    \begin{displaymath}
        - F(\nu) \pt \mathbf{h}_F(M,a)
        = \trace \bigl( \bm{\Xi}(\Der^2 F(\nu)) \circ \bm{\Xi}(\pt \Der^2 (f \circ \project T)(0) \bullet \nu)  \bigr)  \,.
    \end{displaymath}
\end{Definition}

\begin{Remark}
    Note that the above definition does not depend on the choice of~$\nu$
    and~$f$. In~particular, if $\bar \nu = - \nu$, then
    recalling~\eqref{eq:F-one-homog} we obtain
    \begin{gather}
        F(\nu)
        = \langle \nu ,\, \Der F(\nu) \rangle
        = \langle \bar \nu ,\, \Der F(\bar \nu) \rangle
        = F(\bar \nu) \,; \\
        \text{hence,} \quad
        \nu \bigl\langle (\nu, \nu) ,\, \Der^2 F(\nu) \bigr\rangle
        = \nu \bigl\langle (\nu, \bar \nu) ,\, \Der^2 F(\bar \nu) \bigr\rangle
        = \bar \nu \bigl\langle (\bar \nu, \bar \nu) ,\, \Der^2 F(\bar \nu) \bigr\rangle \,.
    \end{gather}
\end{Remark}

\subsection*{Anisotropic nearest point projection and related objects}
\begin{Definition}
    Let $F : \Real{n+1} \to \Real{}$ be a norm, $x \in \Real{n+1} $ and $r > 0$. We define 
    \begin{displaymath}
        \mathbf{U}^{F}(x,r) = \{ y : F(y-x) < r  \}
        \quad \text{and} \quad
        \cballF{F}{x}{r} = \{ y : F(y-x) \leq r  \} \,.
    \end{displaymath}
\end{Definition}

\begin{Definition}\label{Wulff shape}
    Let $F : \Real{n+1} \to \Real{}$ be a~norm. Define the \emph{conjugate norm
      $F^*$} on~$\Real{n+1}$ by setting
    \begin{displaymath}
        F^*(w) = \sup \{ w \bullet u : u \in \Real{n+1} ,\, F(u) \le 1 \} \,.
    \end{displaymath}
    By a~\emph{Wulff shape (of $F$)} we mean any open ball with respect to the
    $F^*$~norm.
\end{Definition}

\begin{Definition}
    Given $ A \subseteq \Real{n+1} $, we define \emph{the anisotropic distance
      function to $ A $} as
    \begin{displaymath}
    \bm{\delta}_{A}^{F}(x) = \inf\{ F^{\ast}(a-x) : a \in A   \} \quad \textrm{for every $ x \in \Real{n+1} $}.
    \end{displaymath}
    Moreover,
    \begin{displaymath}
    S^{F}(A,r) = \{ x : \bm{\delta}_{A}^{F}(x) = r  \} \quad \textrm{for $ r > 0 $.}
    \end{displaymath}
\end{Definition}

\begin{Definition}
    \label{nearest point projection}
    Suppose $ A \subseteq \Real{n} $ is closed and $W$ is the set of all $x \in
    \Real{n}$ such that there exists a~unique $a \in A$ with $F^{\ast}(x-a) =
    \bm{\delta}_{A}^{F}(x)$. The \textit{anisotropic nearest point projection
      onto~$A$} is the map $\bm{\xi}^{F}_{A} : W \to A$ characterised by the
    requirement
    \begin{displaymath}
        F^{\ast}(x- \bm{\xi}^{F}_{A}(x)) = \bm{\delta}_{A}^{F}(x) \quad \textrm{for $x \in W$}.
    \end{displaymath}
    We also define $ \bm{\nu}^{F}_{A}: W \without A \to \partial
    \cballF{F^*}{0}{1}$ and $\bm{\psi}^{F}_{A} : W \without A \to A \times
    \partial \cballF{F^*}{0}{1}$ by the formulas
    \begin{displaymath}
        \bm{\nu}_{A}^{F}(z) = \bm{\delta}_{A}^{F}(z)^{-1}(z -  \bm{\xi}^{F}_{A}(z))
        \quad \text{and} \quad
        \bm{\psi}^{F}_{A}(z)= (\bm{\xi}^{F}_{A}(z), \bm{\nu}^{F}_{A}(z))
        \quad \text{for $z \in W \without A$} \,.
    \end{displaymath}
\end{Definition}

\begin{Definition}[\protect{cf.~\cite[Def.~3.54]{AFP00}}]
    \label{def:red-bdry}
    Let $A \subseteq \Real{n+1}$ be a set of finite perimeter and $V =
    \var{n+1}(A) \in \Var{n+1}(\Real{n+1})$. Then $\| \delta V \|$ is a~Radon
    measure (cf.~\cite[4.7]{Allard1972}) and there exists $\| \delta V
    \|$~measurable function $\bm{\eta}(V,\cdot)$ with values in $\sphere{n}$ as
    in~\cite[4.3]{Allard1972}. We define the~\emph{reduced boundary of~$A$},
    denoted $\partial^{\ast} A$, as the set of points $x \in \dmn
    \bm{\eta}(V,\cdot)$ for which
    \begin{displaymath}
        \| \delta V \|\,\cball xr > 0 \quad \text{for $r > 0$ and} \quad
        \lim_{r \downarrow 0} \frac{1}{\|\delta V\|\, \cball xr}  \int_{\cball xr} \bm{\eta}(V,\cdot) \ud \|\delta V\| = \bm{\eta}(V,x) \,.
    \end{displaymath}
\end{Definition}

In the next lemma we summarize a few facts about relations between~$F$ and~$F^*$.

\begin{Lemma}
    \label{lem:Wulff-as-image}
    Let $F$ be a~strictly convex norm of class~$\cnt{1,1}$, $F^*$ its conjugate,
    $W = \oballF{F}01$, $W^* = \oballF{F^*}01$, $G,G^* : \Real{n+1} \to
    \Real{n+1}$ be given by $G = \grad F$ and $G^* = \grad F^*$.

    The following hold.
    \begin{enumerate}
    \item 
        \label{i:W:norm-G=1}
        $F^*(G(x)) = 1$ and $F(G^*(x)) = 1$ for any $x \in \Real{n+1} \without
        \{0\}$.
    \item
        \label{i:W:G-lip-homeo}
        $G|{\partial W} : \partial W \to \partial W^*$ is a~Lipschitz
        homeomorphism.
    \item
        \label{i:W:F-from-G}
        $F^*(x) = x \bullet G^*(x)$ and $F(x) = x \bullet G(x)$ for $x \in
        \Real{n+1} \without \{0\}$.
    \item
        \label{i:W:F-dbl-conj}
        $F^{**} = F$.
    \item
        \label{i:W:F-strictly-conv}
        $F^*$ is a~strictly convex norm.
    \item
        \label{i:W:G-star-inv-G}
        $G^*|{\partial W^*} = (G|{\partial W})^{-1}$.
    \item
        \label{i:W:F-conj-C1}
        $F^*$ is of class~$\cnt{1}$.
    \item
        \label{i:W:F-conj-C11}
        If $F$ satisfies~\eqref{ell2}, then $F^*$ is of
        class~$\cnt{1,1}$ and $G|{\partial W} : \partial W \to \partial W^*$ is
        bilipschitz.
    \item
        \label{i:W:G-n-id}
        $\mathbf{n}(W,x) = G(x) F(\mathbf{n}(W,x))$ and $\mathbf{n}(W^*,y) =
        G^*(y) F(\mathbf{n}(W^*,y))$ for $x \in \partial W$ and $y \in \partial
        W^*$. In~particular, $G(\mathbf{n}(W^*,y)) = y$ for $y \in \partial
        W^*$ and $G^*(\mathbf{n}(W,x)) = x$ for $x \in \partial W$.
    \end{enumerate}
\end{Lemma}

\begin{proof}
    It is clear from the definition that $F^*$ is a~norm; hence, it is Lipschitz
    and convex. Employ~\cite[Theorem~25.5]{Rockafellar} or the Rademacher
    theorem~\cite[3.1.6]{Federer1969} to see that $F^*$ is differentiable
    $\Leb{n+1}$~almost everywhere. Observe that $G(x) \bullet x = F(x) > 0$ and
    $G^*(y) \bullet y = F^*(y) > 0$ for all $x \in \Real{n+1} \without \{0\}$
    and $y \in \dmn G^*$ due to positive $1$-homogeneity of~$F$ and~$F^*$.

    Assume now that $G|{\partial W}$ is not injective, i.e., that there exist
    $a,b \in \partial W$ such that $a \ne b$ and $G(a) = G(b)$. Since $F(u) =
    F(-u)$ we see that $G(u) = - G(-u)$ so $a \ne -b$ and the line segment
    joining~$a$ and~$b$ does not pass through the origin. Set $u = b-a$ and
    define the strictly convex map $f : [0,1] \to \Real{n+1}$ by the formula
    $f(t) = F(a + tu)$. Then $f'(0) = G(a) \bullet u$ and $f'(1) = G(b) \bullet
    u$ so $f'(0) = f'(1)$ which contradicts strict convexity of~$f$. Therefore
    $G|{\partial W}$ is injective; hence, since $\partial W$ is compact,
    $(G|{\partial W})^{-1} : G[\partial W] \to \partial W$ is continuous.

    For any $w \in \Real{n+1}$ define $g_w : \Real{n+1} \to \Real{}$ by the
    formula $g_w(u) = u \bullet w$ for $u \in \Real{n+1}$.

    Let $w \in \partial W^*$ so that $F^*(w) = 1$. Since the smooth map $g_w$
    attains its maximal value on the compact manifold~$\partial W = F^{-1}\{1\}$
    of class~$\cnt{1,1}$, using the method of Lagrange multipliers, we deduce
    that there exists at least one~$u \in \partial W$ such that $1 = F^*(w) = w
    \bullet u$ and $w = \lambda G(u)$ for some $\lambda \in \Real{}$. Then
    \begin{displaymath}
        1 = u \bullet w
        = u \bullet \lambda G(u)
        = \lambda \langle u ,\, \Der F(u) \rangle
        = \lambda F(u)
        = \lambda \,.
    \end{displaymath} 
    Therefore,
    \begin{displaymath}
        F^*(G(u)) = \lambda F^*(w) = \lambda = 1 \,.
    \end{displaymath}
    Since $w \in \partial W$ was arbitrary, it follows that $\partial W^*
    \subseteq G[\partial W]$. Noting that $\partial W$ and $\partial W^*$ are
    compact connected submanifolds of~$\Real{n+1}$ without boundary of class at
    least~$\cnt{0,1}$ and that $G|{\partial W}$ is a~homeomorphism onto its
    image, employing the invariance of domain theorem, we get $G[\partial W] =
    \partial W^*$.

    For $u \in \partial W$ we have $F^{**}(u) = \sup \{ u \bullet w : F^*(w) = 1
    \} \le 1$, $u \bullet G(u) = F(u) = 1$, and $F^*(G(u)) = 1$, so $F^{**} =
    F$. If~$F(u) = 1$ and $F(u) = F^{**}(u) = u \bullet w$ for some $w \in
    \partial W^* \cap \dmn G^*$, then $u = \lambda G^*(w)$ for some $\lambda \in
    \Real{}$ because the function $g_u$ attains its maximum at~$w$. Hence, the
    same argument as before shows that $F(G^*(w)) = 1$ for all $w \in \partial
    W^* \cap \dmn G^*$.

    If $F^*(y) = 1$, then there exists exactly one $w = (G|{\partial
      W})^{-1}(y) \in \partial W$ for which $F^*(y) = w \bullet y$. If~$y \in
    \partial W^* \cap \dmn G^*$, then we know also that $F^*(y) = G^*(y) \bullet
    y$ and $G^*(y) \in \partial W$; hence, $G^*(y) = (G|{\partial W})^{-1}(y)$.
    Employing~\cite[4.7]{Federer1959} we see that $G^*|{\partial W^*} =
    (G|{\partial W})^{-1}$ and, since $(G|{\partial W})^{-1}$ is continuous,
    $F^*|{\Real{n+1} \without \{0\}}$ is of class~$\cnt{1}$.

    For $u,v \in \Real{n+1} \without \{0\}$ we have
    \begin{multline}
        F^*(u+v) = (u+v) \bullet G^*(u+v)
        = u \bullet G^*(u+v) + v \bullet G^*(u+v)
        \\
        < u \bullet G^*(u) + v \bullet G^*(v)
        = F^*(u) + F^*(v) \,;
    \end{multline}
    because $G^*(u) \ne G^*(v)$, $G^*(u)$ is the unique element of $\partial W$
    which realises $\sup\{ u \bullet w : F(w) = 1 \}$, and $G^*(v)$ is the
    unique element of $\partial W$ which realises $\sup\{ v \bullet w : F(w) = 1
    \}$; hence, $F^*$ is a strictly convex norm.

    Observe that if $x \in \partial W$, $y = G(x)$, $\Der G(x)$ exists, and $v,w
    \in \Tan(\partial W^*, x)$, then
    \begin{displaymath}
        \langle (v,w) ,\, \Der^2F^*(y) \rangle
        = \Der G^*(y)v \bullet w
        = (\Der G(x)|{\Tan(\partial W, x)})^{-1}v \bullet w \,.
    \end{displaymath}
    If $F$ satisfies~\eqref{ell2}, then the right-hand side is bounded and $F^*$
    is of class~$\cnt{1,1}$.

    Finally, note that for $x \in \partial W$ we have $G(x) = \grad F(x) \perp
    \Tan(\partial W,x)$, $G(x) \bullet x = F(x) = 1$, and $F^*(G(x)) = 1$;
    hence, $G(x) = \mathbf{n}(W,x) / F^*(\mathbf{n}(W,x))$. Similarly, $G^*(y) =
    \mathbf{n}(W^*,y) / F(\mathbf{n}(W^*,y))$ whenever $F^*(y) = 1$.
\end{proof}

\begin{Corollary}
    \label{cor:wulff-shape-princ-curvatures}
    Assume $F$ is an elliptic integrand of class~$\cnt{1,1}$, $r \in \Real{}$ is
    positive, $W = \oballF{F^*}0r$, $\eta : \partial W \to \Real{n+1}$ is given
    by $\eta(z) = \grad F(\mathbf{n}(W,z))$ for $z \in \partial W$. We have
    $\eta(z) = z/r$ for $z \in \partial W$ so $\Der \eta(y)v = v/r$ for $v \in
    \Tan(\partial W,y)$ and $y \in \dmn \Der \eta$; hence,
    recalling~\ref{def:F-princ-curv} and~\ref{rem:F-princ-curc-and-mc} we see
    that
    \begin{displaymath}
        \kappa_{\partial W,1}^F(y,\eta(y)) = \ldots = \kappa_{\partial W,n}^F(y,\eta(y)) = 1/r
        \quad \text{for $y \in \dmn \Der \eta$} \,.
    \end{displaymath}
\end{Corollary}

\begin{Lemma}
    \label{lem:xi-diff-ae}
    Let $F : \Real{n+1} \to \Real{}$ be a~strictly convex norm of
    class~$\cnt{1,1}$, $G = \grad F$, $A \subseteq \Real{n+1}$ be closed. Then
    \begin{enumerate}
    \item
        \label{i:xi:delta-1-lip}
        $| \bm{\delta}^F_A(y) - \bm{\delta}^F_A(z)| \le F^*(y-z)$ for $y,z \in \Real{n+1}$.
    \item
        \label{i:xi:xi-cont}
        $\bm{\xi}^F_A$ is continuous.
    \item
        \label{i:xi:delta-on-fiber}
        Suppose $x \in \Real{n+1} \without A$ and $a \in A$ are such that
        $\bm{\delta}^F_A(x) = F^*(x-a)$. Then
        \begin{displaymath}
            \bm{\delta}^F_A(a + t(x-a)) = t F^*(x-a) = t \bm{\delta}^F_A(x)
            \quad \text{for $0 < t \le 1$} \,.
        \end{displaymath}
    \item
        \label{i:xi:xi-by-grad-delta}
        Suppose $x \in \Real{n+1} \without A$ and $a \in A$ are such that
        $\bm{\delta}^F_A(x) = F^*(x-a)$ and $\Der \bm{\delta}^F_A(x)$ exists.
        Then $x \in U$ and $G(\grad \bm{\delta}^F_A(x))\bm{\delta}^F_A(x) =
        x-a$; hence,
        \begin{displaymath}
            \bm{\xi}^F_A(x) = x - G(\grad \bm{\delta}^F_A(x))\bm{\delta}^F_A(x) = a \,.
        \end{displaymath}
    \item 
        \label{i:xi:delta-sqr-C1}
        The maps $\bm{\delta}^F_A|\Int\bigl( \dmn \bm{\xi}^F_A \without A
        \bigr)$ and $(\bm{\delta}^F_A)^2|\Int\bigl( \dmn \bm{\xi}^F_A \bigr)$
        are continuously differentiable and
        \begin{displaymath}
            \bigl\langle u ,\, \Der (\bm{\delta}^F_A)^2(y) \bigr\rangle
            = \bigl\langle u ,\, \Der (F^*)^2(y - \bm{\xi}^F_A(y)) \bigr\rangle
            \quad \text{for $y \in \Int\bigl( \dmn \bm{\xi}^F_A \bigr)$ and $u \in \Real{n+1}$} \,.
        \end{displaymath}
    \item
        \label{i:xi:dmn-xi-ae}
        $\Leb{n+1}(\Real{n+1} \without \dmn \bm{\xi}^F_A) = 0$.
    \item 
        \label{i:xi:delta-xi}
        Assume $a \in A$, $u \in \partial \cballF{F^*}01$, $t > 0$,
        and~$\bm{\delta}^F_A(a + tu) = t$. Then $a+su \in \dmn \bm{\xi}^F_A$ and
        $\bm{\xi}^F_A(a+su) = a$ for all $0 < s < t$. In~particular,
        \begin{displaymath}
            \{ s  : \bm{\xi}^F_A(a+su) = a \}
            \subseteq \{ s : \bm{\delta}^F_A(a+su) = s \}
            = \Clos{ \{ s : \bm{\xi}^F_A(a+su) = a \} } \,.
        \end{displaymath}
    \item 
        \label{i:xi:direction}
        Assume $a \in A$, $x \in \Real{n+1}$, and $\bm{\delta}^F_A(x) =
        F^*(x-a)$. Then
        \begin{displaymath}
            x-a \in G ( \Nor(A,a) ) \,.
        \end{displaymath}
        In~particular, if $\mathbf{n}(A,a) \ne 0$, then
        \begin{displaymath}
            \mathbf{n}^{F}(A,a) = \bm{\nu}^F_A(x) = \frac{x-a}{F^*(x-a)} \,.
        \end{displaymath}
    \end{enumerate}
\end{Lemma}

\begin{proof}
    We mimic parts of the proof of~\cite[4.8]{Federer1959}. We set $G^* = \grad F^*$.
    
    Let $y,z \in \Real{n+1}$, then
    \begin{displaymath}
        \bm{\delta}^F_A(y) \le \bm{\delta}^F_A(z) + F^*(y-z)
        \quad \text{and} \quad
        \bm{\delta}^F_A(z) \le \bm{\delta}^F_A(y) + F^*(y-z) \,;
    \end{displaymath}
    hence, claim~\ref{i:xi:delta-1-lip} follows.

    Assume that~\ref{i:xi:xi-cont} does not hold. Then there are $y_i \in \dmn
    \bm{\xi}^F_A$ for $i \in \nat$ and $\varepsilon > 0$ such that $\lim_{i \to
      \infty} y_i = y \in \dmn \bm{\xi}^F_A$ but $F^*(\bm{\xi}^F_A(y_i) -
    \bm{\xi}^F_A(y)) > \varepsilon$. Using~\ref{i:xi:delta-1-lip} we get
    \begin{displaymath}
        F^*(\bm{\xi}^F_A(y_i) - y) \le \bm{\delta}^F_A(y) + 2 F^*(y_i - y)
        \quad \text{for $i \in \nat$} \,;
    \end{displaymath}
    hence, the set $\{ \bm{\xi}^F_A(y_i) : i \in \nat \}$ is a~bounded subset of
    the closed set~$A$ and we may assume that $\lim_{i \to \infty}
    \bm{\xi}^F_A(y_i) = z \in A$. Then
    \begin{displaymath}
        \bm{\delta}^F_A(y)
        = \lim_{i \to \infty} \bm{\delta}^F_A(y_i)
        = \lim_{i \to \infty}F^*(\bm{\xi}^F_A(y_i)-y_i)
        = F^*(z-y) \,;
    \end{displaymath}
    hence, $\bm{\xi}^F_A(y) = z$ which is incompatible with
    \begin{displaymath}
        F^*(z - \bm{\xi}^F_A(y)) = \lim_{i \to \infty}F^*(\bm{\xi}^F_A(y_i) - \bm{\xi}^F_A(y)) \ge \varepsilon \,.
    \end{displaymath}
    
    Assume \ref{i:xi:delta-on-fiber} does not hold. Then there are $0 < t < 1$
    and $b \in A$ such that setting $y = a + t(x-a)$ we get $F^*(y-b) <
    F^*(y-a)$ and
    \begin{displaymath}
        F^*(x-a) \le F^*(x-b) \le F^*(x-y) + F^*(y-b) < F^*(x-y) + F^*(y-a) = F^*(x-a)
    \end{displaymath}
    a~contradiction.

    Now we prove~\ref{i:xi:xi-by-grad-delta}. We have
    \begin{displaymath}
        \bm{\delta}^F_A(x + t(a-x)) = \bm{\delta}^F_A(x) - t \bm{\delta}^F_A(x)
        \quad \text{for $0 < t < 1$} \,,
    \end{displaymath}
    which implies
    \begin{equation}
        \label{eq:xi:grad-delta}
        \grad \bm{\delta}^F_A(x) \bullet \frac{x-a}{\bm{\delta}^F_A(x)} = \frac{\Der \bm{\delta}^F_A(x)(a-x)}{-\bm{\delta}^F_A(x)} = 1 \,.
    \end{equation}
    From~\eqref{eq:xi:grad-delta} and~\ref{i:xi:delta-1-lip} we conclude
    using~\ref{lem:Wulff-as-image}\ref{i:W:F-dbl-conj}\ref{i:W:F-from-G}
    \begin{multline}
        1 = \sup \{ \Der \bm{\delta}^F_A(x) u : u \in \Real{n+1} ,\, F^*(u) \le 1 \}
        \\
        = \sup \{ \grad \bm{\delta}^F_A(x) \bullet u : u \in \Real{n+1} ,\, F^*(u) \le 1 \}
        \\
        = F^{**}(\grad \bm{\delta}^F_A(x))
        = F(\grad \bm{\delta}^F_A(x))
        = \grad \bm{\delta}^F_A(x) \bullet G(\grad \bm{\delta}^F_A(x)) \,.
    \end{multline}
    However, due to~\ref{lem:Wulff-as-image}\ref{i:W:G-lip-homeo} there is
    exactly one $w \in \Real{n+1}$ with $F^*(w) = 1$ such that $\grad
    \bm{\delta}^F_A(x) \bullet w = F(\grad \bm{\delta}^F_A(x)) = 1$; thus,
    \begin{displaymath}
        G(\grad \bm{\delta}^F_A(x)) = \frac{x-a}{\bm{\delta}^F_A(x)} \,.
    \end{displaymath}

    The formula for $\Der (\bm{\delta}^F_A)^2$ postulated
    in~\ref{i:xi:delta-sqr-C1} is proven exactly as
    in~\cite[4.8(5)]{Federer1959} noting
    \begin{displaymath}
        F^*(\tfrac{x-a}{\bm{\delta}^F_A(x)}) = 1 \,,
        \quad
        \grad \bm{\delta}^F_A(x) = G^*(\tfrac{x-a}{\bm{\delta}^F_A(x)}) \,,
        \quad
        \langle u ,\, \Der (F^*)^2(y) \rangle = 2 F^*(y) G^*(y) \bullet u \,.
    \end{displaymath}
    
    Continuity of the derivatives of~$\bm{\delta}^F_A|\Real{n+1} \without A$
    and~$(\bm{\delta}^F_A)^2$ follows from the formulas and a~reasoning
    completely analogous to the proof of~\cite[4.8(5)]{Federer1959}.

    Item~\ref{i:xi:dmn-xi-ae} is now a~consequence of the~Rademacher
    theorem~\cite[3.1.6]{Federer1969}.

    For the proof of~\ref{i:xi:delta-xi} recall~\ref{i:xi:delta-on-fiber} and
    assume to the contrary, that there exist $0 < s < t$ and $b \in A$, $b \ne
    a$ such that $s = F^*(a + su - a) = F^*(a + su - b) =
    \bm{\delta}^F_A(a+su)$. Set $p = a + su$ and $q = a + tu$. Clearly $b \ne p
    + su$ since otherwise $t = \bm{\delta}^F_A(q) \le F^*(q - b) = F^*(a + tu -
    (a+2su)) = t - 2s < t$ which is impossible. Therefore, $q-a$ and $q-b$ are
    linearly independent and,
    using~\ref{lem:Wulff-as-image}\ref{i:W:F-strictly-conv}, we obtain
    the~contradictory estimate
    \begin{displaymath}
        t \le F^*(q-b) < F*(q-p) + F^*(p-b) = t-s + s = t \,.
    \end{displaymath}

    To prove~\ref{i:xi:direction} we observe that
    \begin{displaymath}
        \oballF{F^*}{x}{F^*(x-a)} \cap A = \varnothing \,;
        \quad \text{hence,} \quad
        -\mathbf{n}(\cballF{F^*}{x}{F^*(x-a)},a) \in \Nor(A,a) \,.
    \end{displaymath}
    Indeed, otherwise there would exist $v \in \Tan(A,a)$ such that $v \bullet
    \mathbf{n}(\cballF{F^*}{x}{F^*(x-a)},a) < 0$ so there would be points $y_i
    \in A$ such that $|y_i - a| \to 0$ and $(y_i - a)/|y_i - a| \to v$ as $i \to
    \infty$ and then, since $F^*$ is of class~$\cnt{1}$, we could find $i \in
    \nat$ for which $y_i \in \oballF{F^*}{x}{F^*(x-a)} \cap A$ and this cannot
    happen. Employing~\ref{lem:Wulff-as-image}\ref{i:W:G-n-id} we see that
    \begin{displaymath}
        G( -\mathbf{n}(\cballF{F^*}{x}{F^*(x-a)},a) ) = \frac{x-a}{F^*(x-a)} \in G(\Nor(A,a)) \,.
        \qedhere
    \end{displaymath}
\end{proof}

\section{Totally umbilical hypersurfaces}\label{tota}

In~\ref{cor:wulff-shape-princ-curvatures} we proved that $\partial
\cballF{F^*}0r$ has all $F$-principal curvatures equal to~$1/r$. In~this section
we show that this condition actually characterises the manifold $\partial
\cballF{F^*}0r$.

\begin{Lemma}
    \label{lem:kappa-constant}
    Suppose $M$ is a~connected submanifold of~$\Real{n+1}$ of class~$\cnt{1,1}$ of
    dimensions~$n$, $\eta : M \to \Real{n+1}$ is Lipschitz, and $\kappa : M \to
    \Real{}$ is such that
    \begin{displaymath}
        \Der \eta(z)(u) = \kappa(z) u
        \quad \text{for $\Haus{n}$ almost all $z \in M$ and all $u \in \Tan(M,z)$} \,.
    \end{displaymath}
    Then $\kappa$ is a~constant function.
\end{Lemma}

\begin{proof}
    Since $M$ is connected it suffices to show the claim only locally. Let $a
    \in M$. We~represent~$M$ near~$a$ as the~graph of some
    $\cnt{1,1}$~function~$f$, i.e., we find $p \in \orthproj{n+1}{n}$, $q \in
    \orthproj{n+1}{1}$, $U \subseteq \Real{n}$ an open ball centred at~$p(a)$,
    and $f : U \to \Real{}$ of class~$\cnt{1,1}$ such that, setting $L = p^* +
    q^* \circ f$, there holds
    \begin{displaymath}
        a \in L[U] \subseteq M
        \quad \text{and} \quad
        q \circ p^* = 0 \,.
    \end{displaymath}
    For each $v \in \Real{n}$ we define
    \begin{displaymath}
        \gamma_v : U \to \Real{}
        \quad \text{by} \quad
        \gamma_v(x) = \eta(L(x)) \bullet v \,.
    \end{displaymath}
    Then
    \begin{multline}
        \Der \gamma_v(x)u
        = \Der \eta(L(x)) ( \Der L(x) u ) \bullet v
        = \kappa(L(x)) (\Der L(x) u) \bullet v
        \\
        = \kappa(L(x)) ( p^*(u) + q^*(\Der f(x)u) ) \bullet v
        = \kappa(L(x)) ( u \bullet p(v) + \Der f(x)u \bullet q(v) )
        \\
        \quad \text{for $\Leb{n}$~almost all $x \in U$, $u \in \Real{n}$, $v \in \Real{n+1}$} \,.
    \end{multline}
    Now, choose an orthonormal basis $e_1,\ldots,e_{n}$ of $\Real{n}$ and set
    $\gamma_i = \gamma_{p^*(e_i)}$ for $i = 1,2,\ldots,n$. Since $q \circ p^* =
    0$ and $p \circ p^* = \mathbf{1}_{\Real{n}}$, we obtain
    \begin{multline}
        \label{eq:Der-gamma-i}
        \Der \gamma_{i}(x) e_j = \kappa(F(x)) ( e_i \bullet e_j ) = 0
        \quad \text{and} \quad
        \Der \gamma_{i}(x) e_i = \kappa(F(x))
        \\
        \quad \text{for $\Leb{n}$~almost all $x \in U$, $i,j \in \{1,2,\ldots,n\}$, and $i \ne j$} \,.
    \end{multline}
    Recall that~$U$ is~an~open ball centred at~$p(a)$. Define $J = \{ (x - p(a))
    \bullet e_1 : x \in U \}$. Since~$\eta$ is~Lipschitz we see that $\gamma_1,
    \ldots, \gamma_{n}$ are~absolutely continuous and deduce
    from~\eqref{eq:Der-gamma-i} that there exist Lipschitz functions
    $a_1,\ldots,a_{n} : J \to \Real{}$ such that
    \begin{multline}
        \gamma_i(x) = a_i( (x - p(a)) \bullet e_i )
        \\ \text{and} \quad
        a_i'( (x - p(a)) \bullet e_i ) = a_j'( (x - p(a)) \bullet e_j ) = \kappa(F(x))
        \\
        \text{for $\Leb{n}$ almost all $x \in U$, $i,j \in \{1,2,\ldots,n\}$} \,.
    \end{multline}
    It follows that $a_i'$ is a~constant function for $i = 1,2,\ldots,n$; hence,
    $\kappa$~is also constant.
\end{proof}

\begin{Lemma}
    \label{lem:umbilical}
    Suppose $F$ is an elliptic integrand of class~$\cnt{1,1}$, $M$~is
    a~connected $n$-dimensional submanifold of~$\Real{n+1}$ of class~$\cnt{1,1}$
    satisfying $\Clos{M} \without M = \varnothing$, $\nu : M \to \Real{n+1}$ is
    Lipschitz and such that $\nu(z) \in \Nor(M,z)$ and $|\nu(z)| = 1$, $\eta : M
    \to \Real{n+1}$ is defined by~$\eta(y) = \grad F( \nu(y) )$, and there
    exists a~scalar function $\kappa : M \to \Real{}$ such that
    \begin{displaymath}
        \Der \eta(y) u = \kappa(y) u
        \quad \text{for $\Haus{n}$ almost all $y \in M$ and all $u \in \Tan(M,y)$} \,.
    \end{displaymath}
    Then there exists $\lambda \in \Real{}$ such that $\kappa(y) = \lambda$ for
    $y \in M$ and either $\lambda = 0$ and $M$ is a~hyperplane in~$\Real{n+1}$
    or $\lambda \ne 0$ and $M = \partial \cballF{F}{a}{|\lambda|^{-1}}$ for some
    $a \in \Real{n+1}$.
\end{Lemma}

\begin{proof}
    In~view of~\ref{lem:kappa-constant} we obtain $\lambda \in \Real{}$ such
    that
    \begin{displaymath}
        \Der \eta(z) u = \lambda u
        \quad \text{for all $\Haus{n}$ almost all $z \in M$ and $u \in \Tan(M,z)$} \,.
    \end{displaymath}
    Therefore, $\Der(\eta - \lambda \id{\Real{n}}) = 0$ and we obtain $c \in
    \Real{n}$ such that
    \begin{displaymath}
        \eta(z) - \lambda z = c
        \quad \text{for all $z \in M$} \,.
    \end{displaymath}
    If $\lambda = 0$, then $\eta$ is constant and $M$ must be a~hyperplane
    because $\Clos{M} \without M = \varnothing$. In~case $\lambda \ne 0$ we set
    $a = - c \lambda^{-1}$ and $\rho = |\lambda|^{-1}$. Then
    \begin{displaymath}
        F^*(z - a)
        = \rho F^*(\eta(z))
        = \rho F^*(\grad F(\nu(z))) = \rho
        \quad \text{for all $z \in M$} \,,
    \end{displaymath}
    by~\ref{lem:Wulff-as-image}\ref{i:W:norm-G=1}. Hence, $M = \partial
    \cballF{F}{a}{\rho}$ because $\Clos{M} \without M = \varnothing$.
\end{proof}

\section{The Lusin property for anisotropic (n,h)-sets}\label{lusi}

In this section $F$ is an~elliptic integrand of class~$\cnt{2}$ and $\Omega
\subseteq \Real{n+1}$ is open.

\begin{Definition}[\protect{cf.~\cite[Definition~3.1]{2019arXiv190103514D}}]
    \label{def:nh-set}
    We say that $Z \subseteq \Omega$ is an~\emph{(n,h)-set with respect to~$F$}
    if $Z$ is relatively closed in~$\Omega$ and for any open set $N \subseteq
    \Omega$ such that $\partial N \cap \Omega$ is smooth and~$Z \subseteq
    \Clos{N}$ there holds
    \begin{displaymath}
        F(\mathbf{n}(N,p)) \mathbf{h}_F(\var{n}(\partial N),p) \bullet \mathbf{n}(N,p) \ge - h
        \quad \text{for $p \in Z \cap \partial N \cap \Omega$} \,.
    \end{displaymath}
\end{Definition}

\begin{Lemma}
    \label{weak maximum principle}
    Suppose $T \in \grass{n+1}{n}$, $\eta \in T^{\perp}$, $|\eta| = 1$, $f : T
    \to T^{\perp}$ is pointwise differentiable of order~$2$ at~$0$ and satisfies
    $f(0)=0$ and $\pt \Der f(0) = 0$, $\Sigma = \{ x + f(x) : x \in T\}$, $h
    \geq 0$, and $\Gamma$ is an~$(n,h)$ subset of~$\Omega$ with respect to~$F$
    such that $0 \in \Gamma$ and
    \begin{displaymath}
        \Gamma \cap V \subseteq \bigl\{ z : z \bullet \eta \leq
        f \circ \project{T}(z) \bullet \eta \bigr\}
    \end{displaymath}
    for some open neighbourhood~$V$ of~$0$. Then
    \begin{displaymath}
        F(\eta) \pt \mathbf{h}_{F}(\Sigma,0) \bullet \eta \geq -h.
    \end{displaymath}
\end{Lemma}

\begin{proof}
    We mimic the proof of~\cite[3.4]{2019arXiv190310379S}. Fix $\varepsilon > 0$,
    define $P,\psi : T \to T^{\perp}$ by
    \begin{gather}
        P(x) = \tfrac 12 \langle (x, x) ,\, \pt \Der^2 f(0) \rangle 
        \quad \text{for $x \in T$} \,,
        \\
        \psi(x) = \bigl( P(x) \bullet \eta + \varepsilon |x|^2 \bigr) \eta
        \quad \text{for $x \in T$} \,,
        \\
        \text{and set} \quad
        M = \Real{n} \cap \bigl\{ x + \psi(x) : x \in T \bigr\} \,.
    \end{gather}
    Note that since $f$ is pointwise differentiable of order~$2$ at~$0$, it 
    follows that
    \begin{displaymath}
        \lim_{x \to 0} \frac{|f(x) - P(x)|}{|x|^{2}} = 0 \,.
    \end{displaymath}
    Hence, we choose $r > 0$ such that $f(x) \bullet \eta \le \psi(x) \bullet
    \eta$ for $x \in \oball 0r \cap T$. Since $\Gamma$ is an $(n,h)$ subset of
    $\Omega$, $M$ is smooth and touches $\Gamma$ at~$0$, and $\Gamma \cap \oball
    0r \subseteq \Real{n+1} \{ x : x \bullet \eta \le \psi(x) \bullet \eta \}$, we
    may use the barrier principle~\cite[Proposition 3.1(iii)]{2019arXiv190103514D}
    to derive the estimate
    \begin{displaymath}
        F(\eta) \mathbf{h}_F(M,0) \bullet \eta \ge -h \,.
    \end{displaymath}
    Recall~\ref{def:ptwise-F-mc} to see that
    \begin{displaymath}
        - F(\eta) \pt \mathbf{h}_{F}(M,0)
        = \eta \trace \bigl( \bm{\Xi}(\Der^2 F(\eta)) \circ \bm{\Xi}(\Der^2 (\psi \circ \project{T})(0) \bullet \eta) \bigr)
        \,.
    \end{displaymath}
    Since
    \begin{displaymath}
        \Der^2 (\psi \circ \project{T})(0)(u,v) \bullet \eta
        = \pt \Der^2 (f \circ \project{T})(0)(u,v) \bullet \eta + 2 \varepsilon u \bullet \project{T}v
        \quad \text{for $u,v \in \Real{n}$} 
    \end{displaymath}
    we see that
    \begin{displaymath}
        - F(0,\eta) \pt \mathbf{h}_{F}(\Sigma,0)
        = - F(\eta) \pt \mathbf{h}_{F}(M,0)
        - 2 \varepsilon \eta \trace \bigl( \bm{\Xi}(\Der^2 F(\eta)) \bigr)
        \,.
    \end{displaymath}
    Passing to the limit $\varepsilon \downarrow 0$ we obtain the claim.
\end{proof}

\begin{Definition}
    \label{Lusin Property}
    Suppose $ A \subseteq \Real{n+1} $ is a closed set. We say that $N(A)$
    satisfies the \textit{$n$~dimensional Lusin~(N) condition in~$\Omega$} if
    and only if
    \begin{displaymath}
        S \subseteq A \cap \Omega
        \quad \text{and} \quad
        \Haus{n}(S) = 0
        \quad \text{implies that} \quad
        \Haus{n}(N(A)|S) = 0 \,.
    \end{displaymath}
\end{Definition}

\begin{Theorem}
    \label{Lusin property for (m,h) sets}
    Suppose $0 \leq h < \infty$, $A$ is an~$(n,h)$ subset of~$\Omega$ with
    respect to~$F$ that is a~countable union of~sets with finite~$\Haus{n}$
    measure.
    
    Then $N(A)$ satisfies the $n$ dimensional Lusin~(N) condition in~$\Omega$.
\end{Theorem}

\begin{proof}
    We modify the proof of~\cite[3.7]{2019arXiv190310379S}. Let $\tau > \lambda
    = 2 C(F)^2(n-1) + 1$, where $C(F)>0$ is defined in~\eqref{def:CF}.

    \medskip
    
    Claim 1: \emph{Assume $r \in \Real{}$ satisfies $0 \leq h <
      \frac{1}{2C(F)r}$, and $ x\in S(A,r) \cap R(A) \cap A_{\tau} \cap
      \bm{\xi}_{A}^{-1}(A) $ (see \ref{regular ponints}) is such that
      $\Hdensity{n}{S(A,r) \without A_{\tau}}{x}=0$, and the conclusions of
      \cite[Lemma 2.8]{2019arXiv190310379S} are satisfied. Consider an
      orthonormal basis $v_1,\dots,v_{n+1}$ in which the matrix of~$\ap \Der
      \bm{\nu}_{A}(x)$ is diagonal and $v_{n+1} = \bm{\nu}_{A}(x)$. We introduce
      abbreviations
      \begin{displaymath}
          \partial_{ij} F(\nu) = \langle (v_i,v_j) ,\, \Der^2 F(\nu) \rangle
          \quad \text{for $i,j \in \{1,2,\ldots,n+1\}$} \,.
      \end{displaymath}
      Then we have
      \begin{displaymath}
          \sum_{i=1}^{n}\partial_{ii}F(\bm{\nu}_{A}(x))\chi_{A,i}(x) \leq h 
          \quad \text{and} \quad
          \| {\textstyle \bigwedge_{n}} \big( (\Haus{n}\restrict S(A,r),n) \ap \Der \bm{\xi}_{A}(x)\big) \| > 0 \,.
      \end{displaymath}
    }

    \medskip
    
    Noting that $\bm{\xi}_{A}|A_{\lambda}$ is approximately differentiable
    at~$x$ (since $x \in R(A)$), we
    employ~\cite[3.8, 3.10(3)(6)]{2017arXiv170801549S}
    and~\cite[3.2.16]{Federer1969} to conclude that
    \begin{gather}
        \label{Lusin property for (m,h) sets:2}
        \chi_{A,j}(x) \geq -(\lambda -1)^{-1}r^{-1}
        \quad \textrm{for $ j = 1, \ldots , n $} \,,
        \\
        \label{Lusin property for (m,h) sets:4}
        \ap \Der \bm{\xi}_{A}(x)| \Tan(\Haus{n}\restrict S(A,r),x) = (\Haus{n}
        \restrict S(A,r), n)\ap \Der \bm{\xi}_{A}(x).
    \end{gather}
    We choose $f$, $V$ and $T$ as in \cite[Lemma 2.8]{2019arXiv190310379S} and
    $0 < s < r/2 $ such that $\oball{x}{s} \subseteq V$. We assume
    $\bm{\xi}_{A}(x) = 0 \in \Gamma$ and we notice that $\project{T}(x) = 0$ and
    $\bm{\nu}_{A}(x) = r^{-1}x$. Then we define $g(\zeta) = f(\zeta) - x$ for
    $\zeta \in T$,
    \begin{displaymath}
        U = \project{T} \big( \mathbf{U}(x,s) \cap \{ \chi + f(\chi) : \chi \in T \} \big),
        \quad W = \{ y-x : y \in \project{T}^{-1}(U) \cap \mathbf{U}(x,s) \} \,.
    \end{displaymath}
    It follows that $W$ is an open neighbourhood of~$0$ and
    \begin{equation}
        \label{Lusin property for (m,h) sets:5}
        W \cap A \subseteq
        \{ z : z \bullet \bm{\nu}_{A}(x) \leq g( \project{T}(z) ) \bullet \bm{\nu}_{A}(x) \} \,.
    \end{equation}
    Indeed, if \eqref{Lusin property for (m,h) sets:5} did not hold, then there
    would be $y \in \mathbf{U}(x,s) \cap \project{T}^{-1}[U]$ such that $y-x \in
    A$ and $y \bullet \bm{\nu}_{A}(x) > f(\project{T}(y)) \bullet
    \bm{\nu}_{A}(x)$; noting that
    \begin{displaymath}
        \project{T}(y) + f(\project{T}(y)) \in \mathbf{U}(x,s) \cap S(A,r)
        \quad \text{and} \quad
        |\project{T}(y) + f(\project{T}(y))-y| < r \,,
    \end{displaymath}
    we would conclude
    \begin{displaymath}
        |\project{T}(y) + f(\project{T}(y))-(y-x)|
        = r - (y-f(\project{T}(y))) \bullet \bm{\nu}_{A}(x)
        < r = \bm{\delta}_A( \project{T}(y) + f(\project{T}(y)) )
    \end{displaymath}
    which is a contradiction.

    Since $ -\chi_{A,1}(x), \ldots , - \chi_{A,n}(x) $ are the eigenvalues of
    $\pt \Der^{2} g(0) \bullet \bm{\nu}_{A}(x) $ and $0 \in A$, we~may
    apply~\ref{weak maximum principle} to infer that
    \begin{equation}
        \label{Lusin property for (m,h) sets:3}
        \partial_{11}F(\bm{\nu}_{A}(x)) \chi_{A,1}(x)
        + \ldots
        + \partial_{nn}F(\bm{\nu}_{A}(x)) \chi_{A,n}(x) \leq h
    \end{equation}
    and combining~\eqref{ell2}, \ref{def:CF}, \eqref{Lusin property for (m,h)
      sets:2}, and~\eqref{Lusin property for (m,h) sets:3} we get that for every
    $j = 1, \ldots , n$
    \begin{multline}
        \chi_{A,j}(x)
        \leq C(F)\partial_{jj}F(\bm{\nu}_{A}(x))\chi_{A,j}(x)
        \leq C(F)h-C(F)\sum_{k\neq j,k=1}^n\partial_{kk}F(\bm{\nu}_{A}(x))\chi_{A,k}(x)
        \\
        \leq C(F)h-\frac{C(F)^2(n-1)}{(\lambda -1)r}<\frac{1}{r} \,.
    \end{multline}
    From~\eqref{Lusin property for (m,h) sets:4}
    and~\cite[3.5]{2017arXiv170801549S} follows that $1-r\chi_{A,j}(x)$ are the
    eigenvalues of $(\Haus{n} \restrict S(A,r), n)\ap \Der \bm{\xi}_{A}(x)$ for
    $j = 1, \ldots , n$; hence, we obtain
    \begin{displaymath}
        \bigl\| {\textstyle \bigwedge_{n}} \bigl( (\Haus{n}\restrict S(A,r),n) \ap \Der \bm{\xi}_{A}(x) \bigr) \bigr\|
        \geq \prod_{i=1}^{n}\big(1-\chi_{A,i}(x)r\big) > 0 \,. 
    \end{displaymath}

    \medskip
    
    Claim 2: \emph{For $\Haus{n}$ a.e. $ x \in S(A,r) \cap A_{\tau} \cap
      \bm{\xi}_{A}^{-1}(A) $ and for $ \Leb{1} $ a.e.\ $ 0 < r <
      \frac{1}{2C(F)h} $ the conclusion of Claim 1 holds.}

    \medskip
    
    This is immediate since
    \begin{displaymath}
        \Hdensity{n}{S(A,r)\without A_{\tau}}{x} =0
    \end{displaymath} for $ \Haus{n} $ a.e.\ $ x \in S(A,r) \cap A_{\tau} $ and
    for every $ r > 0 $ by \cite[2.13(1)]{2017arXiv170801549S} and
    \cite[2.10.19(4)]{Federer1969}, and $ \Haus{n}(S(A,r)\without R(A)) =0 $ for $
    \Leb{1} $ a.e.\ $ r > 0 $ by \cite[3.16]{2017arXiv170801549S}.

    \medskip
    
    Claim 3: \emph{$N(A)$ satisfies the $n$ dimensional Lusin~(N) condition in
      $\Omega$.}

    \medskip
    
    Let $R \subseteq A$ be such that $\Haus{n}(R) = 0$. For $r > 0$ it
    follows from \cite[3.17, 3.18(1), 4.3]{2017arXiv170801549S} that
    $\bm{\psi}_{A}|A_{\tau} \cap S(A,r)$ is a~bilipschitz homeomorphism and
    \begin{displaymath}
        \bm{\psi}_{A}(\bm{\xi}_{A}^{-1}\{x\} \cap A_{\tau} \cap S(A,r) )
        \subseteq N(A,x)
        \quad \textrm{for $ x \in A $} \,.
    \end{displaymath}
    Noting Claim 2 and \cite[3.10(1)]{2017arXiv170801549S}, we can apply
    \cite[Lemma 3.5]{2019arXiv190310379S} with $W$ and $f$ replaced by $S(A,r)
    \cap A_{\tau} \cap \bm{\xi}_{A}^{-1}(A) $ and $ \bm{\xi}_{A}|S(A,r) \cap
    A_{\tau} \cap \bm{\xi}_{A}^{-1}(A) $ to infer that
    \begin{displaymath}
        \Haus{n}(\bm{\xi}_{A}^{-1}(R) \cap S(A,r) \cap A_{\tau} ) =0 \quad
        \textrm{for $ \Leb{1} $ a.e.\ $ 0 < r < \frac{1}{2C(F)}h^{-1} $.}
    \end{displaymath}
    We notice that $ N(A)|R = \bigcup_{r>0}\bm{\psi}_{A}(S(A,r)\cap A_{\tau}
    \cap \bm{\xi}_{A}^{-1}(R)) $ by \cite[4.3]{2017arXiv170801549S} and $
    \bm{\psi}_{A}(S(A,r)\cap A_{\tau}) \subseteq \bm{\psi}_{A}(S(A,s) \cap
    A_{\tau}) $ if $ s < r $ by \cite[3.18(2)]{2017arXiv170801549S}. Henceforth,
    it follows that
    \begin{displaymath}
        \Haus{n}(N(A)|R) = 0 \,.
        \qedhere
    \end{displaymath}
\end{proof}

The following weak maximum principle is a simple consequence of~\cite[Theorem
3.4]{2019arXiv190103514D}.
\begin{Lemma}
    \label{mh}
    Assume
    \begin{displaymath}
        V \in \Var{n}(\Omega) \,,
        \quad
        F(\overline{\mathbf{h}}_F(V,x)) \le h \quad \text{for $\|V\|$~almost all $x$} \,,
        \quad
        \| \delta_F V\|_{\mathrm{sing}} = 0 \,.
    \end{displaymath}
    Then $\spt \|V\|$ is an $(n, h)$ subset of~$\Omega$ with respect to~$F$.
\end{Lemma}

\begin{proof}
    For every $k \in \mathbb N$ let $V_k = k \cdot V$. Note that
    \begin{displaymath}
        u \bullet v = \frac{u}{F(u)} \bullet \frac{v}{F^*(v)} F(u) F^*(v)
        \le F(u) F^*(v)
        \quad \text{whenever $u,v \in \Real{n+1}$} \,;
    \end{displaymath}
    thus,
    \begin{displaymath}
        \delta_F V_k(g)
        = - \int \overline{\mathbf{h}}_F(V,x) \bullet g(x) \ud \|V_k\|(x)
        \le h \int F^*(g(x)) \ud \|V_k\|(x)
        \quad \text{for $k \in \nat$ and $g \in \VF(\Omega)$} \,.
    \end{displaymath}
    Moreover, the area blowup set
    \begin{displaymath}
        Z = \bigl\{
        x \in \Clos{\Omega} :
        \limsup_{k \to \infty} \|V_k\|(\cball{x}{r}) = +\infty \, \text{ for every $r>0$}
        \bigr\}
    \end{displaymath}
    coincides with $\spt \|V\|$; hence, \cite[Theorem 3.4]{2019arXiv190103514D}
    yields that $\spt \|V\| = Z$ is an $(n, h)$ set.
\end{proof}

\section{The anisotropic unit normal bundle}\label{anis}
In this section we will need to work with a suitable anisotropic variant of the
normal bundle for closed sets. Let us introduce some definitions.

\begin{Definition}
    \label{normal bundel anisotropo}
    Suppose $F$ is an~elliptic integrand and $A \subseteq \Real{n+1}$ is closed.
    The \emph{generalized anisotropic unit normal bundle of $A$} is defined as
    \begin{displaymath}
        N^F(A) = (A \times \partial \cballF{F^*}01)
        \cap \{ (a,u) : \bm{\delta}^{F}_{A}(a + su) = s \; \textrm{for some $ s > 0 $} \} \,.
    \end{displaymath} 
\end{Definition}

\begin{Lemma}
    \label{normal bundel isotropo vs anisotropo}
    Suppose $F$ is an~elliptic integrand of class~$\cnt{1,1}$ and $A \subseteq
    \Real{n+1}$ is closed. Then
    \begin{displaymath}
        N^F(A) = (\id{\Real{n+1}} \times \grad F) [ N(A) ]
        = \bigl\{ (a, \grad F(u)) : (a,u) \in N(A) \bigr\} \,.
    \end{displaymath} 
    In particular, $N^{F}(A)$ is a~countably $n$~rectifiable Borel subset of
    $\Real{n+1} \times \partial \cballF{F^*}01$.
\end{Lemma}
\begin{proof}
    Given $(a,u) \in N^{F}(A)$, there exists $s > 0$ such that
    \begin{displaymath}
        a \in A \cap \partial \oballF{F^*}{a+su}{s}
        \quad \text{and} \quad
        \oballF{F^*}{a+su}{s} \cap A = \varnothing \,.
    \end{displaymath}
    Since $\partial \oballF{F^*}{a+su}{s}$ is submanifold of~$\Real{n+1}$ of
    class~$\cnt{1,1}$ (see~\ref{lem:Wulff-as-image}\ref{i:W:F-conj-C11}), there
    exists $r > 0$ and $x \in \Real{n+1}$ such that $\oball xr \subseteq
    \oballF{F^*}{a+su}{s}$ and $a \in \partial \oball xr$. It follows that
    \begin{displaymath}
        \mathbf{n}(\oball xr,a) = \mathbf{n}(\oballF{F^*}{a+su}{s},a)
        \quad \text{and} \quad
        (a,-\mathbf{n}(\oballF{F^*}{a+su}{s},a)) \in N(A) \,.
    \end{displaymath}
    Since $\grad F(\mathbf{n}(\oballF{F^*}01,z)) = z$ for every $z \in \partial
    \oballF{F^*}01$ (see~\ref{lem:Wulff-as-image}\ref{i:W:G-n-id}), it follows
    that
    \begin{displaymath}
        \grad F \bigl( -\mathbf{n}(\oballF{F^*}{a+su}{s},a) \bigr)
        = - \grad F \bigl( \mathbf{n}(\oballF{F^*}{a+su}{s},a) \bigr)
        = - \frac{a-(a+su)}{s}
        = u \,,
    \end{displaymath}
    i.e. $(a,u) \in (\id{\Real{n+1}} \times \grad F)(N(A))$.
    
    The proof of the reverse inclusion $(\id{\Real{n+1}} \times \grad F)(N(A))
    \subseteq N^F(A)$ is completely analogous and the postscript follows from
    \cite[4.3]{2017arXiv170801549S}.
\end{proof}

\begin{Definition}
    \label{anisotropic Lusin Property}
    Suppose $ \Omega \subseteq \Real{n+1} $ is open, $F$ is an~elliptic
    integrand, and $A \subseteq \Real{n+1}$ is closed. We say that $N^{F}(A)$
    satisfies the \emph{$n$ dimensional Lusin~(N) condition in $\Omega$} if and
    only if the following implication holds,
    \begin{displaymath}
        S \subseteq A \cap \Omega, \quad \Haus{n}(S) = 0
        \quad \implies \quad \Haus{n}(N^{F}(A)|S) = 0 \,.
    \end{displaymath}
\end{Definition}

\begin{Lemma}
    \label{anisotropic Lusin}
    Assume $F$ is an~elliptic integrand of class~$\cnt{1,1}$, $\Omega \subseteq
    \Real{n+1}$ is open, and $A \subseteq \Real{n+1}$ is closed. Then $N(A)$
    satisfies the $n$~dimensional Lusin~(N) condition in~$\Omega$ if and only if
    $N^{F}(A)$ satisfies the $n$~dimensional Lusin~(N) condition in~$\Omega$.
\end{Lemma}

\begin{proof}
    Let $S \subseteq A \cap \Omega$ be such that $\Haus{n}(S) = 0$. Assume that
    either $\Haus{n}(N^F(A)|S) = 0$ or $\Haus{n}(N(A)|S) = 0$. Since the map
    $\id{\Real{n+1}} \times \grad F$ is a~bilipschitz homeomorphism
    (see~\ref{lem:Wulff-as-image}\ref{i:W:F-conj-C11}), we deduce that
    $\Haus{n}(N^F(A)|S) = \Haus{n}( (\id{\Real{n+1}} \times \grad F)(N(A)|S) ) =
    \Haus{n}(N(A)|S) = 0$ as desired.
\end{proof}

\begin{Definition}
    Let $F$ be an~elliptic integrand and $A \subseteq \Real{n+1}$ be closed.
    The \emph{anisotropic reach function} $ r_{A}^{F} : N^{F}(A) \rightarrow [0
    ,\infty] $ is defined by
    \begin{displaymath}
        r_{A}^{F}(a,u) = \sup \{ s : \bm{\delta}_{A}^{F}(a+su) = s  \}
        \quad \text{for $(a,u) \in N^{F}(A)$} \,.
    \end{displaymath}
    
    The \emph{anisotropic reach of $A$} is defined by
    \begin{displaymath}
        \reach^{F}(A)
        = \inf\bigl\{ \sup \{ r : \oballF{F^*}ar \subseteq \dmn \bm{\xi}_{A}^{F}  \} : a \in A \bigr\}
        = \sup \bigl\{ r : \{ x : \bm{\delta}^F_A(x) < r \} \subseteq \dmn \bm{\xi}^F_A \bigr\} 
        \,.
    \end{displaymath}
\end{Definition}

\begin{Remark}
    \label{rem:rFA-Borel}
    Since $\bm{\delta}^F_A$ is Lipschitz continuous
    (see~\ref{lem:Wulff-as-image}\ref{i:xi:delta-1-lip}), the function $f_s :
    N^F(A) \to \Real{}$ given by $f_s(a,u) = \bm{\delta}^F_A(a + su) - s$ is
    also Lipschitz for any $s \in \Real{}$. Therefore $r^F_A$ is
    lower-semicontinuous. In~particular, $r^F_A$ is a~Borel function.
\end{Remark}

\begin{Lemma}
    \label{cor:reach-submanifold-of-class-11}
    Suppose $F$ is an elliptic integrand of class~$\cnt{1,1}$ and $A$ is
    a~closed submanifold of~$\Real{n+1}$ of class~$\cnt{1}$ such that $\reach^F
    A > 0$. Then $\reach A > 0$ and $A$ is a~submanifold of~$\Real{n+1}$ of
    class~$\cnt{1,1}$.
\end{Lemma}

\begin{proof}
    Set $W = \cballF{F^*}01$. First observe that $\partial W$ is a~submanifold
    of $\Real{n+1}$ of class~$\cnt{1,1}$
    by~\ref{lem:Wulff-as-image}\ref{i:W:F-conj-C11}. Therefore, there exists
    $\rho \in (0,1)$ such that for each $x \in \partial W$ we have
    \begin{displaymath}
        \cball{x + \rho \mathbf{n}(W,x)}{\rho} \subseteq W \,.
    \end{displaymath}

    Assume $\reach^F A = s > 0$. Let $z \in \Real{n+1}$ be such that
    $\bm{\delta}_A(z) = r < \rho s$ and find $x \in A$ with $|z-x| =
    \bm{\delta}_A(z)$. Set $B = \cball zr$, $u = -\mathbf{n}(B,x)$, and $w = x +
    r \grad F(u) / \rho$. Note that $u \in \Tan(A,x)^{\perp}$. We have
    $\bm{\delta}^F_A(w) = r /\rho < s$ so $w \in \dmn \bm{\xi}^F_A$ and
    $\cballF{F^*}{w}{r/\rho} \cap A = \{x\}$ and $\cball zr \subseteq
    \cballF{F^*}{w}{r/\rho}$; hence, $z \in \dmn \bm{\xi}_A$.

    Since $z$ was arbitrary we see that $\{ x : \bm{\delta}_A(x) < \rho s \}
    \subseteq \dmn \bm{\xi}_A$ which shows that $\reach A \ge \rho s$. The~second
    part of the conclusion readily follows from~\cite[4.20]{Federer1959}.
\end{proof}

\begin{Corollary}
    \label{cor:level-set-of-class-11}
    Suppose $A \subseteq \Real{n+1}$ is closed and $\reach^F A > 0$. Then
    $S^F(A,r)$ is a~submanifold of~$\Real{n+1}$ of class~$\cnt{1,1}$ of
    dimension~$n$ for every $0 < r < \reach^F A$.
\end{Corollary}

\begin{proof}
    Since $R = \reach^{F} A > 0$, we have that $\Real{n+1} \cap \{ y :
    \bm{\delta}^F_A(y) < R \} \subseteq \dmn \bm{\xi}^F_A$. Therefore,
    from~\ref{lem:xi-diff-ae}\ref{i:xi:delta-sqr-C1}\ref{i:xi:xi-by-grad-delta}
    and~\ref{lem:Wulff-as-image}\ref{i:W:G-star-inv-G} it follows that
    $\bm{\delta}^F_{A}|\Real{n+1} \cap \{ y : 0 < \bm{\delta}^F_A(y) < R \}$ is
    of class~$\cnt{1}$ and
    \begin{displaymath}
        \grad \bm{\delta}^F_A(y)
        = \grad F^* \Bigl( \frac{ x - \bm{\xi}^F_A(y) }{ \bm{\delta}^F_A(y) } \Bigr)
        \ne 0
        \quad \text{for $y \in \Real{n+1}$ with $0 < \bm{\delta}^F_A(y) < R$} \,.
    \end{displaymath}
    Consequently, for every $0 < r < R$ we see that $S^F(A,r) = (
    \bm{\delta}^F_A )^{-1}\{r\}$ is a~closed submanifold of~$\Omega$ of
    class~$\cnt{1}$ of dimension~$n$. Moreover, we have $\reach^F S^F(A,r) \ge
    \min \{ R-r , r \} > 0$ so the conclusion follows
    from~\ref{cor:reach-submanifold-of-class-11}.
\end{proof}

We prove now the anisotropic version of \cite[Theorem~3]{MR2036332}, whose proof
is essentially along the same lines.

\begin{Theorem}
    \label{sufficient condition for positive reach}
    Assume $F$ is an~elliptic integrand of class~$\cnt{1,1}$ and $A \subseteq
    \Real{n+1}$ is closed.  Let $r >0 $ and suppose that for every
    $\Haus{n}$~measurable bounded function $f : \Real{n+1} \times \partial
    \oballF{F^*}01 \rightarrow \Real{}$ with compact support there are numbers
    $c_{1}(f), \ldots , c_{n+1}(f) \in \Real{}$ such that
    \begin{equation}
        \label{woww}
        \int_{\Real{n+1} \without A} f \circ \bm{\psi}_{A}^{F} \cdot \CF{ \{ x : \bm{\delta}^{F}_{A}(x) \leq t \} } \ud \Leb{n+1}
        = \sum_{j=1}^{n+1} c_{j}(f) t^{j}
        \quad \text{for $0 < t < r $} \,.
    \end{equation}
    Then $\reach^{F}(A) \geq r$.
\end{Theorem}

\begin{proof}
    Let $ S =  \{ (x,u,t) : (x,u) \in N^F(A) ,\, r^{F}_{A}(x,u)> t \} $ and define $\phi
    : N^{F}(A) \times (0, \infty) \rightarrow \Real{n+1} $
    \begin{displaymath}
        \phi(x,u,t) = x + tu \quad \textrm{for $ (x,u,t) \in N^{F}(A) \times (0,\infty) $.}
    \end{displaymath}

    \medskip
    
    Claim 1: \textit{$\Leb{n+1}(\dmn \bm{\xi}^F_A \without (A \cup \phi(S))) = 0$; hence,
      \begin{displaymath}
          \Leb{n+1}(\Real{n+1} \without (A \cup \phi(S))) = 0 \,.
      \end{displaymath}
    }

    \medskip
    
    Recalling~\ref{lem:Wulff-as-image}\ref{i:xi:delta-xi} we see that
    \begin{displaymath}
        \dmn \bm{\xi}^F_A \without (A \cup \phi(S))
        = \phi( \{ (x,u,t) : (x,u) \in N^F(A) ,\, t = r^F_A(x,u) > 0 \} ) \,.
    \end{displaymath}
    Since $\phi$ is a~locally Lipschitz map, it suffices to prove that
    \begin{equation}
        \label{eq:meas-of-t=r}
        \Haus{n+1}(\{ (x,u,t) : (x,u) \in K ,\, M > t = r^F_A(x,u) > 0 \}) = 0
    \end{equation}
    for all $M \in \nat$ and $K \subseteq N^F(A)$ bounded. By~\ref{normal bundel
      isotropo vs anisotropo} and~\cite[3.2.29]{Federer1969} we know that
    $N^F(A)$ is countably $n$~rectifiable. Hence, it suffices to
    prove~\eqref{eq:meas-of-t=r} for all $M \in \nat$ and $K \subseteq A$ being
    $n$~rectifiable. Assume $K$ and $M$ are such.
    Employing~\cite[3.2.23]{Federer1969} we get
    \begin{equation}
        \label{eq:Haus-K-times-M}
        \Haus{n+1}( K \times (0,M+1) ) = (M+1) \Haus{n}(K) < \infty \,.
    \end{equation}
    Recall~\ref{rem:rFA-Borel}. For $q \in \Real{}$ define the Borel set
    \begin{displaymath}
        V_q = \{ (x,u,t+q) : (x,u) \in K ,\, M > t = r^F_A(x,u) > 0 \}
    \end{displaymath}
    and observe that
    \begin{gather}
        V_q \cap V_p = \varnothing \quad \text{whenever $p \ne q$} \,,
        \quad
        V_q \subseteq K \times (0,M+1)\quad \text{for $0 < q < 1$} \,,
        \\
        \quad \text{and} \quad
        \Haus{n+1}(V_q) = \Haus{n+1}(V_0) \quad \text{for any $q \in \Real{}$} \,.
    \end{gather}
    Therefore, if $\Haus{n+1}(V_0) > 0$, then $\Haus{n+1}(\bigcup\{ V_q : 0 < q
    < 1 ,\, q \text{ rational} \}) = \infty$ which
    contradicts~\eqref{eq:Haus-K-times-M}.

    \medskip
    
    Claim 2: \textit{
      \begin{equation}
          \label{sufficient condition for positive reach: eq1}
          \Leb{n+1} \big(
          \{ z : 0 < \bm{\delta}_{A}^{F}(z) \leq r ,\; r_{A}^{F}(\bm{\psi}_{A}^{F}(z) ) < r  \}
          \big) = 0 \,.
      \end{equation}
    } 
    
    \medskip
    
    In the following sequence of estimates we have to deal with the problem that
    $N^{F}(A)$ might not have locally finite measure so $\mu = \Haus{n}
    \restrict N^F(A)$ might not be Radon and $(\mu,n)$~approximate derivative
    of~$\phi$ might not be well defined.

    Recalling~\ref{lem:Wulff-as-image}\ref{i:xi:delta-xi} one readily infers
    that $\phi|S $ is injective.
    Since $N^{F}(A)$ is Borel and countably
    $n$~rectifiable (see~\ref{normal bundel isotropo vs anisotropo}) we may find
    a partition
    \begin{displaymath}
        N^{F}(A) = {\textstyle \bigcup_{i=1}^\infty} N_i
    \end{displaymath}
    such that each $N_i$ is a~Borel $n$~rectifiable set (in~particular,
    $\Haus{n}(N_i) < \infty$) and the family $\{ N_i : i \in \nat \}$ is
    disjointed; cf.~\cite[2.1.6]{Federer1969}. For $i \in \nat$ w define
    \begin{displaymath}
        \mu_i = \Haus{n} \restrict N_i \,,
        \quad
        S_i = S \cap \bigl( N_i \times (0,\infty) \bigr) \,,
        \quad \text{and} \quad
        J = \sum_{i=1}^{\infty} \big\|{\textstyle\bigwedge_{n}}[(\mu_i, n)\ap \Der \phi] \big\| \CF{S_i} \,.
    \end{displaymath}
    We apply Claim~1 and the coarea formula~\cite[3.2.22]{Federer1969} to find
    that
    \begin{multline}
        \label{11}
        \int_{\Real{n+1} \without A} g \ud \Leb{n+1}
        = \int_{\phi(S)} g \ud \Leb{n+1}
        = \sum_{i=1}^{\infty} \int_{\phi(S_i)} g \ud \Leb{n+1}
        \\
        = \int_{0}^{\infty} \sum_{i=1}^{\infty} \int_{N_i}
        \bigl\| {\textstyle \bigwedge_{n}}[(\mu_i, n)\ap \Der \phi(x,u,t)] \bigr\|
        g(x+tu)
        \CF{\{ (w,v) : r_{A}^{F}(w,v) > t\}}(x,u) \ud \Haus{n}(x,u) \ud t
        \\
        =  \int_{0}^{\infty}
        J(x,u,t) g(x+tu)
        \CF{\{ (w,v) : r_{A}^{F}(w,v) > t\}}(x,u) \ud \Haus{n}(x,u) \ud t 
    \end{multline}
    whenever $g : \Real{n+1} \to \Real{}$ is a non-negative Borel function with
    compact support.
    
    Let $B \subseteq \Real{n+1}$ be compact, $0 < \tau < r$ and $\tau < t< r$. We
    define
    \begin{displaymath}
        N_{\tau,B} = N^{F}(A) \cap \{(x,u): r^{F}_{A}(x,u) \leq \tau, \; x \in B  \} \,, 
    \end{displaymath}
    and we apply \eqref{woww} to the function $\CF{N_{\tau,B}}$ and \eqref{11}
    to the function $g= (\CF{N_{\tau,B}}\circ \bm{\psi}_{A}^{F})\cdot \CF{\{w :
      \bm{\delta}^{F}_{A}(w) \leq t\}}$ to compute
    \begin{multline}
        \label{22}
        \sum_{j = 1}^{n+1}c_{j}(f)t^{j}
        \overset{\eqref{woww}}{=} \int_{\Real{n+1} \without A}
        \CF{N_{\tau,B}}(\bm{\psi}_{A}^{F}(z))
        \CF{\{w : \bm{\delta}^{F}_{A}(w) \leq t\}}(z) \ud \Leb{n+1}z
        \\
        \overset{\eqref{11}}{=} \int_{0}^{\infty} \int_{N^{F}(A)}
        J(x,u,s) \CF{\{w : \bm{\delta}^{F}_{A}(w) \leq t\}}(x+su)
        \CF{\{(w,v) : r_{A}^{F}(w,v) > s\}}(x,u)
        \CF{N_{\tau,B}}(\bm{\psi}_{A}^{F}(x+su)) \ud \Haus{n}(x,u) \ud s
        \\
        = \int_{0}^{\infty} \int_{N^{F}(A)}
        J(x,u,s) \CF{\{w : \bm{\delta}^{F}_{A}(w) \leq t\}}(x+su)
        \CF{\{(w,v) : r_{A}^{F}(w,v) > s\}}(x,u)
        \CF{N_{\tau,B}}(x,u) \ud \Haus{n}(x,u) \ud s
        \\
        = \int_{0}^{\infty} \int_{N^{F}(A)}
        J(x,u,s) \CF{\{w: \bm{\delta}^{F}_{A}(w) \leq t\}}(x+su)
        \CF{\{(w,v) : s < r_{A}^{F}(w,v) \leq \tau\}}(x,u)
        \CF{B}(x) \ud \Haus{n}(x,u) \ud s
        \\
        = \int_{0}^{\infty} \int_{N^{F}(A)}
        J(x,u,s) \CF{\{(w,v) : s < r_{A}^{F}(w,v) \leq \tau\}}(x,u)
        \CF{B}(x) \ud \Haus{n}(x,u) \ud s \,,
    \end{multline}
    where the last equality follows because
    $\bm{\delta}^{F}_{A}(x+su)=s<r_{A}^{F}(x,u)\leq \tau<t$, for every
    $\tau<t<r$.  Whence, we deduce that $\sum_{j = 1}^{n+1}c_{j}(f)t^{j}$ is
    independent of~$t$, for every $\tau < t < r$. Therefore, this polynomial is
    identically zero, a~condition that implies, by the first equality in
    \eqref{22},
    \begin{displaymath}
        \Leb{n+1}\big(
        \{ z : 0 < \bm{\delta}_{A}^{F}(z) \leq r ,\; \bm{\psi}_{A}^{F}(z) \in N_{\tau, B}  \}
        \big) = 0 \,.
    \end{displaymath}
    Since the last equation holds for every $0 < \tau < r$ and for every compact
    set $B \subseteq \Real{n+1}$, we conclude that~\eqref{sufficient condition
      for positive reach: eq1} holds.

    \medskip
    
    Claim 3: \textit{$\reach^{F}(A) \geq r $.}

    \medskip
    
    Let $z \in \Real{n+1} \without A$ satisfy $0 < \bm{\delta}_{A}^{F}(z) <
    r$. Then there exists a sequence $\{z_{i} : i \in \nat\} \subseteq \dmn
    \bm{\xi}_{A}^{F} $ which converges to $z$ and such that
    \begin{displaymath}
        0 < \bm{\delta}_{A}^{F}(z_{i}) \leq r
        \quad \text{and} \quad
        r_{A}^{F}(\bm{\psi}_{A}^{F}(z_{i}) ) \geq r \,.
    \end{displaymath}
    Noting that $ (\bm{\xi}_{A}^{F}(z_{i})) $ is a~bounded sequence, and passing to
    a~subsequence if necessary, we find $p \in A$ and $u \in \partial
    \oballF{F^*}01$ such that
    \begin{displaymath}
        \bm{\xi}_{A}^{F}(z_{i}) \to p, \qquad \bm{\nu}_{A}^{F}(z_{i}) \to u \,.
    \end{displaymath}
    In particular, $z = p + \bm{\delta}_{A}^{F}(z)u$. We find $t \in \Real{}$ such
    that $\bm{\delta}_{A}^{F}(z) < t < r $, and notice that
    \begin{equation}
        \label{eq:reach-hhl}
        \oballF{F^*}{\bm{\xi}^{F}_{A}(z_{i}) + t\bm{\nu}_{A}^{F}(z_{i})}{t} \cap A = \varnothing
        \quad \text{for $ i \geq 1 $} \,;
        \quad \text{hence,} \quad
        \oballF{F^*}{p + tu}{t} \cap A = \varnothing \,.
    \end{equation}
    This shows that $\bm{\delta}^F_A(p+tu) = t > \bm{\delta}_{A}^{F}(z)$;
    hence,~\ref{lem:Wulff-as-image}\ref{i:xi:delta-xi} yields $z \in \dmn
    \bm{\xi}^F_A$ and~$\bm{\xi}^F_A(z) = p$.
\end{proof}

\section{Heintze Karcher inequality}
\label{secmain}

Here we prove our main theorem~\ref{Heintze-Karcher}.

\begin{Remark}
    \label{rem:C1a-to-C2a}
    Let $F$ be an~elliptic integrand. Recalling~\cite[5.1.1]{Federer1969} we
    define $\Phi : \Real{n+1} \times \tbwedge_n \Real{n+1} \to \Real{}$,
    a~parametric integrand of degree~$n$ on~$\Real{n+1}$, by setting
    \begin{displaymath}
        \Phi(z,\xi) = F(* \xi)
        \quad \text{for $z \in \Real{n+1}$ and $\xi \in \tbwedge_n \Real{n+1}$} \,,
    \end{displaymath}
    where $*$ denotes the Hodge star operator associated with the standard
    scalar product and orientation on~$\Real{n+1}$;
    see~\cite[1.7.8]{Federer1969}. By~\ref{def:elliptic}
    and~\cite[5.1.2]{Federer1969} we see that~$\Phi$ is elliptic in the sense
    of~\cite[5.1.2]{Federer1969}. Moreover, if~$\Phi^{\S}$ is the nonparametric
    integrand associated with~$\Phi$ (see~\cite[5.1.9]{Federer1969}) and
    $\Phi^{\S}_z(\xi) = \Phi^{\S}(z,\xi)$ for $(z,\xi) \in \Real{n+1} \times
    \tbwedge_n \Real{n+1}$, then $\Der^2 \Phi^{\S}_z(\xi)$ is strongly elliptic
    in the sense of~\cite[5.2.3]{Federer1969} for all $(z,\xi) \in \Real{n+1}
    \times \tbwedge_n \Real{n+1}$ by~\cite[5.2.17]{Federer1969}.

    Let $W \subseteq \Real{n}$ be open and bounded, $V \in \Var{n}(W \times
    \Real{})$, $\mathbf{p} : \Real{n+1} \to \Real{n}$ and $\mathbf{q} :
    \Real{n+1} \to \Real{}$ be given by $\mathbf{p}(z_1,\ldots,z_{n+1}) =
    (z_1,\ldots,z_n)$ and $\mathbf{q}(z_1,\ldots,z_{n+1}) = z_{n+1}$ for
    $(z_1,\ldots,z_{n+1}) \in \Real{n+1}$. Assume $f : \Real{n} \to \Real{}$ is
    of class~$\cnt{1}$, and $V$ is the unit density varifold associated to the
    graph of~$f$, i.e., $V = \var{n}(\im (\mathbf{p}^* + \mathbf{q}^* \circ
    f))$.  Recalling~\cite[5.1.9]{Federer1969} we see that for any $\theta : W
    \to \Real{}$ of class~$\cnt{1}$ with compact support there holds
    \begin{displaymath}
        \delta_F V(\mathbf{q}^* \circ \theta \circ \mathbf{p})
        = \int \bigl\langle
        (0,\theta(x),\Der \theta(x))  ,\, \Der \Phi^{\S}(x,f(x),\Der f(x))
        \bigr\rangle \ud \Leb{n+1}(x) \,.
    \end{displaymath}

    Suppose $F$ is of class~$\cnt{3}$, $\alpha \in (0,1)$, $f$ is of
    class~$\cnt{1,\alpha}$, $\| \delta_F V \|$ is a~Radon measure, $\|\delta_F
    V\|_{\mathrm{sing}} = 0$, and $\mathbf{h}_F(V,\cdot) : \spt\|V\| \to
    \Real{n+1}$ is of class~$\cnt{0,\alpha}$. Define $\eta : W \to \Real{n+1}$
    and $H : W \to \Real{}$ by the formulas
    \begin{gather}
        \eta(x) = (\mathbf{q}^*(1) - \mathbf{p}^*(\grad f(x))) \cdot (1 + |\grad f(x)|^2)^{-1/2} 
        \\
        \text{and} \quad
        H(x)
        = - F( \eta(x) )
        \cdot \mathbf{q} \circ \mathbf{h}_F(V, (\mathbf{p}^* + \mathbf{q}^* \circ f)(x))
        \cdot \sqrt{1 + |\grad f|^2} 
    \end{gather}
    for $x \in W$. Note that $\eta(x)$ is the unit normal vector to the graph
    of~$f$ at~$(\mathbf{p}^* + \mathbf{q}^* \circ f)(x)$ for $x \in W$.
    Employing the area formula~\cite[3.2.3]{Federer1969} we get
    \begin{displaymath}
        \delta_F V(\mathbf{q}^* \circ \theta \circ \mathbf{p})
        = - \int_{\spt \|V\|} \theta( \mathbf{p}(z) ) \cdot \mathbf{q}( \mathbf{h}(V,z) ) \cdot F(\eta(\mathbf{p}(z)) \ud \Haus{n}(z)
        = \int_{W} \theta(x) \cdot H(x) \ud \Leb{n}(x)
    \end{displaymath}
    so that 
    \begin{multline}
        \label{eq:amc-holder}
        \int_W \bigl\langle
        (0,\theta(x),\Der \theta(x))  ,\, \Der \Phi^{\S}(x,f(x),\Der f(x))
        \bigr\rangle \ud \Leb{n+1}(x)
        = \int_W \theta(x) \cdot H(x) \ud \Leb{n}(x)
        \\ \text{for any $\theta \in \dspace{W}{\Real{}}$} \,.
    \end{multline}
    Since $H$ is of class~$\cnt{0,\alpha}$ a slight modification of the proof
    of~\cite[5.2.15]{Federer1969} shows that $f$ is actually of
    class~$\cnt{2,\alpha}$.

    To support the last claim recall the proof of~\cite[5.2.15]{Federer1969}
    with $2$, $n+1$, $n$, $\alpha$, $W$, $\Phi^{\S}$ in place of $q$, $n$, $m$,
    $\delta$, $U$, $G$. Using all the symbols defined therein, for any integer
    $\nu$ such that $\nu > 1/d$, define $R_{\nu} : \cball{b}{\rho-d} \to
    \Hom(\Real{n},\Real{})$ so that
    \begin{displaymath}
        \sigma \bullet R_{\nu}(x) = \int_0^1 \sigma(e_i) \cdot H(x + t e_i / \nu) \ud \Leb{1}(t)
        \quad \text{for $\sigma \in \Hom(\Real{n},\Real{})$ and $x \in \cball{b}{\rho-d}$} \,.
    \end{displaymath}
    Since, in our case, $f$ satisfies~\eqref{eq:amc-holder} rather
    than~\cite[5.2.15(4)]{Federer1969} the displayed equation in the middle of
    page 556 of~\cite{Federer1969}, i.e.,
    \begin{displaymath}
        \int_{\oball{b}{\rho-d}} \bigl\langle \Der f_{\nu}(x) \odot \Der \theta(x) ,\, A_{\nu}(x) \bigr\rangle \ud \Leb{n}(x)
        = \bm{(} P_{\nu} - Q_{\nu} ,\, \Der \theta \bm{)}_{b,\rho-d}
    \end{displaymath}
    turns into
    \begin{displaymath}
        \int_{\oball{b}{\rho-d}} \bigl\langle \Der f_{\nu}(x) \odot \Der \theta(x) ,\, A_{\nu}(x) \bigr\rangle \ud \Leb{n}(x)
        = \bm{(} P_{\nu} - Q_{\nu} - R_{\nu} ,\, \Der \theta \bm{)}_{b,\rho-d} \,.
    \end{displaymath}
    Clearly $R_{\nu}$ is $\alpha$-H{\"o}lder continuous with H{\"o}lder constant
    independent of~$\nu$ so all the estimates from the upper half of page 557
    of~\cite{Federer1969} hold in the modified case with an additional term
    coming from~$R_{\nu}$. Thus, one can still use~\cite[5.2.2]{Federer1969} to
    conclude that $\Der_i f$ is of class~$\cnt{1,\alpha}$; hence, $f$ is of
    class~$\cnt{2,\alpha}$.
\end{Remark}

\begin{Remark}
    \label{rem:reduced-and-essential-bdry}
    Suppose $E \subseteq \Real{n+1}$ is of finite perimeter. We recall that the
    reduced boundary (see~\ref{def:red-bdry}) and the essential boundary
    (cf.~\cite[4.5.12]{Federer1969} and~\cite[Def.~3.60]{AFP00}) of~$E$ are
    $\Haus{n}$~almost the same
    (see~\cite[Thm.~3.61]{AFP00}). Recalling~\cite[4.7]{Allard1972} we~deduce
    that $\mathbf{n}(E,\cdot)|\partial^{\ast} E : \partial^{\ast} E \to
    \Real{n+1}$ equals the negative of the \emph{generalised inner normal
      to~$E$} defined in~\cite[Def.~3.54]{AFP00}.
\end{Remark}

\begin{Definition}
    Let $A \subseteq \Real{n+1}$, $k \in \nat$, $\alpha \in [0,1]$. We say that
    $x \in A$ is a~\emph{$\cnt{k,\alpha}$-regular point of~$A$} if there exists
    an open set $W \subseteq \Real{n+1}$ such that $x \in W$ and $A \cap W$ is
    an~$n$-dimensional submanifold of class~$\cnt{k,\alpha}$ of $\Real{n+1}$.
    The set of all $\cnt{k,\alpha}$~regular points of~$A$ shall be called
    the~\emph{$\cnt{k,\alpha}$~regular part of~$A$}.
\end{Definition}

The strategy for the proof of our main theorem can be summarised in the
following way. First we replace the set~$E$ with an open set~$\Omega$ with the
same essential boundary using~\cite[2.2]{2019arXiv190805952S}. Using standard regularity
theory for codimension one varifolds with bounded anisotropic mean
curvature~\cite{AllardReg} and~\ref{rem:C1a-to-C2a} we deduce that
$\Haus{n}$~almost all of $\partial^* \Omega$ is $\cnt{2,\alpha}$~regular. On the
$\cnt{2,\alpha}$~regular part we can express the, variationally defined,
anisotropic mean curvature vector $\mathbf{h}_F(V,\cdot)$ as the trace of the
anisotropic second fundamental form as in~\ref{rem:Fmc-for-smooth}. However,
this does not reduce the problem to the smooth case because we have no control
of the singular set and we do not know how different parts of the regular set
are arranged in space. Therefore, we look at level-sets $S^F(C,r)$ of the
anisotropic distance function from $C = \Real{n+1} \without \Omega$. These sets
are easily seen to be~$\cnt{1,1}$ submanifolds of~$\Omega$ of dimension~$n$ so
we gain a~priori regularity. Nonetheless, we need to transfer the information we
have from $\Haus{n}$~almost all of~$\partial^* \Omega$ onto $S^F(C,r)$ and then
back to $\partial^* \Omega$. To this end we need the Lusin~(N) condition for
$\partial^* \Omega$ which follows from the weak maximum principle~\ref{mh}
and~\ref{Lusin property for (m,h) sets}. The Lusin~(N) property of $\partial
\Omega$ allows to represent $\Leb{n+1}$~almost all of~$\Omega$ as the image of
the map $\zeta(x,t) = x + t \mathbf{n}^F(C,x)$, where $x$ belongs to the
\emph{regular part of $\partial C$} and $t > 0$ is bounded by the first
eigenvalue of the anisotropic second fundamental form of $\partial \Omega$
at~$x$. At this point we apply the Montiel-Ros argument to estimate the measure
of~$\Omega$ and derive the Heintze-Karcher inequality.

Next, we deal with the equality case. First we note that the principal
curvatures of~$\partial \Omega$ must all equal $-n/H(z)$ for $z$ in the regular
part of~$\partial \Omega$. We use the Steiner formula~\ref{sufficient condition
  for positive reach} to deduce that $\reach^F C > n/c$. Then we let $0 < r <
n/c$ and we compute the principal curvatures of the level-set $S^F(C,r)$ using
the information we have on the regular part of~$\partial \Omega$. This and the
Lusin~(N) property show that $S^F(C,r)$ is totally umbilical
at~$\Haus{n}$~almost all points. Since we know that $S^F(C,r)$ is of
class~$\cnt{1,1}$, the $\Haus{n}$~almost everywhere information is enough to
apply~\ref{lem:umbilical} to see that $S^F(C,r)$ is a~finite union of boundaries
of Wulff shapes of radii $n/c - r$. After that, it is rather easy to see that
each connected component of~$\Omega$ must be a~Wulff shape of radius at
least~$n/c$. Since the perimeter of~$\Omega$ is finite we see also that there
may be at most finitely many connected components of~$\Omega$.

\begin{Theorem}
    \label{Heintze-Karcher}
    Suppose 
    \begin{gather}
        \text{$F$ is an~elliptic integrand of class~$\cnt{3}$} \,,
        \quad
        n \ge 2 \,,
        \quad
        c \in (0,\infty) \,,
        \\
        \text{$E \subseteq \Real{n+1} $ is a set of finite perimeter} \,,
        \quad
        \Haus{n}\bigl( \Clos(\partial^{\ast} E) \without \partial^{\ast}E \bigr) = 0 \,,
        \\
        V = \var{n}(\partial^{\ast}E) \in \RVar{n}(\Real{n+1}) \,,
        \quad
        \| \delta_F V \|_{\mathrm{sing}} = 0 \,,
        \\
        \text{$\mathbf{h}_F(V,\cdot)|K$ is of class~$\cnt{0,\alpha}$ for each compact subset $K$ of the $\cnt{1,\alpha}$~regular part of~$\spt \|V\|$} \,,
        \\
        0  < - \overline{\mathbf{h}}_F(V,x) \bullet \mathbf{n}(E,x) \le c
        \quad \text{for $\|V\|$~almost all $x$}
        \,.
    \end{gather}
    Then 
    \begin{equation}
        \label{Heintze-Karcher:eq1}
        \Leb{n+1}(E) \leq \frac{n}{n+1} \int_{\partial E} \frac{1}{| \mathbf{h}_F(V,x) |} \ud \Haus{n}(x)
    \end{equation}
    and equality holds if and only if here there exists a finite union $\Omega$
    of disjoint open Wulff shapes with radii not smaller than $n/c$ such that
    $\Leb{n+1}\bigl( (\Omega \without E) \cup (E \without \Omega) \bigr) = 0$.
\end{Theorem}

\begin{proof}
    First we employ~\cite[2.2]{2019arXiv190805952S} to obtain an open set $\Omega
    \subseteq \Real{n+1}$ such that
    \begin{displaymath}
        \Leb{n+1}\bigl( (\Omega \without E) \cup (E \without \Omega) \bigr) = 0
        \quad \text{and} \quad
        \Haus{n}(\partial \Omega \without \partial^* \Omega) = 0 \,.
    \end{displaymath}
    Directly from the definition (see~\cite[4.5.12, 4.5.11]{Federer1969})
    it~follows that the essential boundaries of~$\Omega$ and~$E$ coincide;
    hence, recalling~\ref{rem:reduced-and-essential-bdry}, we obtain $V =
    \var{n}(\partial^{\ast} \Omega)$. We shall consider~$\Omega$ instead of~$E$
    in the sequel. Let us define
    \begin{gather}
        H : \spt \|V\| \to [0,c]
        \quad \text{so that} \quad
        H(x) = - \overline{\mathbf{h}}_F(V,x) \bullet \mathbf{n}(E,x)
        \quad \text{for $\|V\|$~almost all~$x$} \,,
        \\
        C = \Real{n+1} \without \Omega \,,
        \quad
        Q = \partial C \cap \bigl\{ x : \text{$x$ is a $\cnt{2,\alpha}$-regular point of $\partial C$} \bigr\} \,.
    \end{gather}
    Note that $\partial^{\ast}C = \partial^{\ast}\Omega$,
    $\mathbf{n}^{F}(C,\cdot) = - \mathbf{n}^{F}(\Omega,\cdot)$, and
    $H(x) = F(\mathbf{n}(E,x)) |\mathbf{h}_F(V,x)|$ for $\|V\|$~almost
    all~$x$.

    \medskip

    Claim 1: \emph{If $x \in Q$, $y \in \Omega$, and $\bm{\xi}^F_{C}(y) = x$ (in
      other words: $y \in \Omega \cap (\bm{\xi}^F_{C})^{-1}(Q)$), then}
    \begin{displaymath}
        0 \le \frac 1n H(x)
        \le - \kappa_{Q,1}^F(\bm{\psi}_C^F(y))
        \le \bm{\delta}^F_C(y)^{-1} \,.
    \end{displaymath}

    \medskip
    
    We clearly have
    \begin{displaymath}
        \oballF{F^*}{y}{\bm{\delta}^F_C(y)} \cap C = \varnothing 
        \quad \text{and} \quad
        \partial \oballF{F^*}{y}{\bm{\delta}^F_C(y)} \cap C = \{ x \} \,;
    \end{displaymath}
    hence, recalling~\ref{cor:wulff-shape-princ-curvatures},
    \ref{rem:F-princ-curc-and-mc}, and that $x$ is a~$\cnt{2,\alpha}$-regular
    point of $\partial C$, wee see that
    \begin{displaymath}
        \frac 1n H(x)
        \le - \kappa_{Q,1}^F(\bm{\psi}_C^F(y))
        \le - \kappa_{\partial \oballF{F^*}{y}{\bm{\delta}^F_C(y)},1}^F(\bm{\psi}_C^F(y))
        = \bm{\delta}^F_C(y)^{-1} 
    \end{displaymath}
    and the claim is proven.
    
    \medskip
    
    Claim 2: \emph{$\Leb{n+1}(\Omega \without (\bm{\xi}_{C}^F)^{-1}(Q)) = 0$.}

    \medskip

    Note that $F(\overline{\mathbf{h}}_F(V,x)) = H(x) F(\mathbf{n}(\Omega,x))$
    for $\|V\|$ almost all $x$ so applying Lemma~\ref{mh} we conclude that
    $\partial \Omega$ is an $(n,c\,C(F))$ subset of~$\Real{n+1}$. It follows by
    Theorem \ref{Lusin property for (m,h) sets} that $\Haus{n}(N(\partial
    \Omega)|S) =0$ whenever $S \subseteq \Real{n+1}$ satisfies $ \Haus{n}(S) =0
    $. Combining this with Lemma \ref{anisotropic Lusin}, we deduce that
    $\Haus{n}(N^F(\partial \Omega)|S) =0$ whenever $ S \subseteq \Real{n+1} $
    satisfies $ \Haus{n}(S) =0 $. Since $ N^F(C) \subseteq N^F(\partial \Omega)
    $, one readily infers that $ \Haus{n}(N^F(C)|S) =0 $ whenever $ S \subseteq
    \Real{n+1} $ satisfies $\Haus{n}(S) = 0$. We also observe that for $\|V\|$
    almost all~$z$ there exists a~radius~$r > 0$ such that $V$ satisfies all the
    assumption of \cite[The Regularity Theorem, pp. 27-28]{AllardReg} inside
    $\oball zr$. This implies that there exists $\alpha \in (0,1)$ such that for
    $\Haus{n}$ almost all $z \in \partial C$ there exists an open set $G \subset
    \mathbb R^{n+1}$ with $z \in G$ and such that $\partial C \cap G$ coincides
    with a~rotated graph of some function $f : \Real{n} \to \Real{}$ of
    class~$\cnt{1,\alpha}$. However, employing~\ref{rem:C1a-to-C2a}, we see that
    $f$ is actually of class~$\cnt{2,\alpha}$.  Therefore,
    \begin{equation}
        \label{Heintze-Karcher:eq2}
        \Haus{n}(\partial C \without Q) = 0
        \quad \text{and} \quad
        \Haus{n}(N^F(C)|(\partial C \without Q )) = 0 \,.
    \end{equation}
    Since $ \bm{\psi^F}_{C}\big(S^F(C,r) \cap (\dmn \bm{\xi}^F_{C}) \without
    (\bm{\xi}^F_{C})^{-1}(Q) \big) \subseteq N(C)|(\partial C \without Q ) $ for
    every $ r > 0 $, we get
    \begin{displaymath}
        \Haus{n}\big(
        \bm{\psi^F}_{C}\big(
        S^F(C,r) \cap (\dmn \bm{\xi}^F_{C}) \without (\bm{\xi}^F_{C})^{-1}(Q)
        \big)
        \big) = 0
        \quad \textrm{for every $ r > 0 $} \,.
    \end{displaymath}
    Moreover, we have $ \big(\bm{\psi^F}_{C}|(S^F(C,r) \cap \dmn \bm{\xi}^F_{C}
    \without C)\big)^{-1} \in \cnt{1}$ and we deduce that
    \begin{displaymath}
        \Haus{n}\big( S^F(C,r) \cap (\dmn \bm{\xi}^F_{C}) \without (\bm{\xi}^F_{C})^{-1}(Q)\big) =0 \quad \textrm{for every $ r > 0 $.}
    \end{displaymath}
    Combining~\ref{lem:xi-diff-ae}\ref{i:xi:dmn-xi-ae}\ref{i:xi:delta-1-lip}\ref{i:xi:xi-by-grad-delta}
    with the coarea formula~\cite[3.2.22]{Federer1969}, we get
    \begin{displaymath}
        \Haus{n}(S^F(C,r) \without \dmn  \bm{\xi}^F_{C}   ) = 0
        \quad \text{for $\Leb{1}$ almost all $r > 0$} \,.
    \end{displaymath}
    From~\ref{lem:xi-diff-ae}\ref{i:xi:xi-by-grad-delta} it follows that
    $F(\grad \bm{\delta}_{C}^{F}(x)) = 1$; hence, recalling~\ref{def:CF},
    we~obtain $|\grad \bm{\delta}_{C}^{F}(x)|\geq \frac{1}{C(F)}$. Using the
    coarea formula, we compute
    \begin{multline}
        \frac{1}{C(F)}\Leb{n+1}(\Omega \without (\bm{\xi}^F_{C})^{-1}(Q))
        \\
        \leq \int_{\Omega \without (\bm{\xi}^F_{C})^{-1}(Q)} |\grad \bm{\delta}_{C}^{F}(x)| dx
        = \int_{0}^{\infty}\Haus{n}(S^F(C,r) \without (\bm{\xi}^F_{C})^{-1}(Q)) \, dr
        = 0 \,.
    \end{multline}
    In particular we get that $\Leb{n+1}(\Omega \without
    (\bm{\xi}^F_{C})^{-1}(Q)) = 0$, which settles Claim~2.

    We define
    \begin{gather}
        Z = (Q \times \Real{}) \cap \bigl\{ (x,t) :  0 < t \le -\kappa^F_{Q,1}(x,\mathbf{n}^F(C,x))^{-1} \bigr\} \,,
        \\
        \zeta : Z \to \Real{n+1} \,, \quad \zeta(x,t) = x + t \mathbf{n}^{F}(C,x) \,.
    \end{gather}
    For brevity of the notation we also set
    \begin{displaymath}
        J_{n+1}\zeta(x,t) = \| \tbwedge_{n+1} (\Haus{n+1} \restrict Z,n+1) \ap \Der \zeta(x,t) \| 
        \quad \text{whenever $(x,t) \in Z$} \,.
    \end{displaymath}

    \medskip
    
    Claim 3: \emph{There holds}
    \begin{equation}
        \label{utilissima}
        J_{n+1}\zeta(x,t) = F(\mathbf{n}(C,x)) \prod_{i=1}^n \bigl( 1 + t\,\kappa^F_{Q,i}(x,\mathbf{n}^F(C,x) ) \bigr)
        \quad \textit{for $(x,t) \in Z$} \,.
    \end{equation}

    \medskip
    
    Let $(x,t) \in Z$ and $u =
    \mathbf{n}^F(C,x)$. Recalling~\ref{rem:two-self-adjoint} we find a~basis
    $\tau_1(x), \ldots, \tau_n(x)$ of~$\Tan(Q,x)$ consisting of eigenvectors
    of~$\Der ( \mathbf{n}^F(C,\cdot) )(x)$ and such that
    \begin{gather}
        \bigl\langle \tau_i(x) ,\, \Der \mathbf{n}^F(C,\cdot)(x) \bigr\rangle
        = \kappa^F_{Q,i}(x,u)\,\tau_i(x)
        \quad \text{for $i \in \{1,2,\ldots,n\}$} \,,
        \\
        |\tau_1(x) \wedge \cdots \wedge \tau_n(x)| = 1 \,.
    \end{gather}
    Noting that $\Tan(Z, (x,t)) = \Tan(Q,x) \times \mathbf R$,
    \begin{gather}
        \bigl\langle (0,1) ,\, \Der \zeta(x,t) \bigr\rangle = \mathbf{n}^F(C,x) = \grad F( \mathbf{n}(C,x) ) \,,
        \\
        \bigl\langle (\tau_{i}(x),0) ,\, \Der \zeta(x,t) \bigr\rangle 
        = (1 + t \kappa^F_{Q,i}(x,u))\,\tau_i(x)
        \quad \text{for $i \in \{ 1, \ldots , n \}$} \,,
    \end{gather}
    we compute
    \begin{multline}
        J_{n+1}\zeta(x,t)
        = \prod_{i=1}^{n}(1+t \kappa^F_{Q,i}(x,u))\, | \mathbf{n}^F(C,x) \wedge \tau_1(x) \wedge \cdots \wedge \tau_n(x) |
        \\ 
        =  \grad F(\mathbf{n}(C,x)) \bullet \mathbf{n}(C,x) \,
        \prod_{i=1}^{n}(1+t \kappa^F_{Q,i}(x,u))\,
        | \mathbf{n}(C,x) \wedge \tau_1(x) \wedge \cdots \wedge \tau_n(x) |
    \end{multline}
    and Claim~3 follows from~\ref{lem:Wulff-as-image}\ref{i:W:F-from-G} and~\cite[1.7.5]{Federer1969}.

    \medskip
    
    Claim 4: \emph{Inequality \eqref{Heintze-Karcher:eq1} holds.}

    \medskip

    Employing Claim~1 and Claim~2 we see that $\Leb{n+1}( \Omega \without
    \zeta(Z) ) = 0$. Hence, using the area formula and then Claim~3, we~get
    \begin{multline}
        \label{primo}
        \Leb{n+1}(\Omega) \le \Leb{n+1}(\zeta(Z))
        \le \int_{\zeta({Z})} \Haus{0}(\zeta^{-1}(y)) \ud \Leb{n+1}(y)
        = \int_{Z} J_{n+1} \zeta \ud \Haus{n+1}
        \\
        = \int_{Q} F(\mathbf{n}(C,x))
        \int_0^{-1/\kappa^F_{Q,1}(x,\mathbf{n}^F(C,x))} \prod_{i=1}^n \bigl( 1 + t\kappa^F_{Q,i}( x, \mathbf{n}^F(C,x) ) \bigr)
        \ud t  \ud \Haus{n}(x)
        \,.
    \end{multline}
    Using again Claim~1, then the standard inequality between the arithmetic and
    the geometric mean, and finally~\ref{rem:F-princ-curc-and-mc}, we obtain
    \begin{multline}
        \label{anisotropic montiel ros}
        \Leb{n+1}(\Omega)
        \le \int_{Q} F(\mathbf{n}(C,x))
        \int_0^{-1/\kappa^F_{Q,1}(x,\mathbf{n}^F(C,x))} \Bigl( \frac 1n \sum_{i=1}^n \bigl( 1 + t \kappa^F_{Q,i}( x, \mathbf{n}^F(C,x) ) \bigr) \Bigr)^n
        \ud t \ud \Haus{n}(x)
        \\
        \le \int_{Q} F(\mathbf{n}(C,x)) \int_0^{n/H(x)} \Bigl( 1 - t\frac{H(x)}{n} \Bigr)^n \ud t \ud \Haus{n}(x) \\
        =\frac{n}{n+1}\,\int_{\partial\Omega} \frac{F(\mathbf{n}(C,x))}{H(x)} \ud \Haus{n}(x)\,,
    \end{multline}
    which implies \eqref{Heintze-Karcher:eq1} by~\ref{def:F-mean-curvature}.

    We assume now that equality holds in \eqref{Heintze-Karcher:eq1}. Since the
    chains of inequalities \eqref{primo} and \eqref{anisotropic montiel ros}
    become chains of equalities, we deduce that
    \begin{gather}
        \label{Heintze-Karcher:eq4}
        \Leb{n+1}(\zeta({Z}) \without \Omega) = 0 \,,
        \\
        \label{Heintze-Karcher:eq5}
        \Haus{0}(\zeta^{-1}(y)) = 1
        \quad \text{for $ \Leb{n+1} $ almost all $ y \in \zeta(Z)$} \,,
        \\
        \label{Heintze-Karcher:eq6}
        -\kappa^F_{Q,j}(z, \mathbf{n}^F(C,z))^{-1} = \frac{n}{H(z)}
        \quad \text{for $\Haus{n}$ almost all $z \in Q$ and all $j = 1, \ldots, n$} \,.
    \end{gather}
    Our goal is to prove that $ \Omega$ is a~finite union of disjoint open Wulff
    shapes. We need two preliminary claims, whence the conclusion will be easily
    deduced.

    Claim 5: \emph{$ \reach^{F} C \geq n/c $.}

    Recall that $H(z) \le c$ for $\Haus{n}$ almost all $z \in \partial C$. Let
    $0 < \rho < n/c$ and
    \begin{displaymath}
        Q_{\rho} = Q \cap \{ z : \rho < -\kappa^F_{Q,1}(z,\mathbf{n}^F(C,z))^{-1}  \}.
    \end{displaymath}
    It follows from \eqref{Heintze-Karcher:eq2}, \eqref{Heintze-Karcher:eq6},
    and the fact that $\partial C$ is an $(n,c\,C(F))$ subset of~$\Real{n+1}$,
    that
    \begin{displaymath}
        \Haus{n}(\partial C \without Q_{\rho}) = 0
        \quad \text{and} \quad
        \Haus{n}(N(C)| \partial C \without Q_{\rho}) = 0 \,;
    \end{displaymath}
    hence, we argue as in Claim 2 to conclude that $\Leb{n+1}(\Omega \without
    \bm{\xi}_{C}^{-1}(Q_{\rho})) = 0$. We define
    \begin{displaymath}
        C_{\rho}^F = \{ z : \bm{\delta}^F_{C}(z) \leq \rho  \}
        \quad \text{and} \quad Z_{\rho} = Q_{\rho} \times \{ t : 0 < t \leq \rho  \}
    \end{displaymath}
    and we notice that
    \begin{displaymath}
        \bm{\xi}_{C}^{-1}(Q_{\rho}) \cap \Omega \cap C_{\rho}^F \subseteq \zeta(Z_{\rho}) \subseteq C_{\rho}^F \,,
        \quad
        \Leb{n+1}(\Omega \cap C_{\rho}^F \without \zeta(Z_{\rho})) = 0 \,.
    \end{displaymath}
    Let $ f : \Real{n+1} \times \sphere{n} \rightarrow \Real{} $ be a Borel
    measurable function with compact support. Then we use Claim~1,
    \eqref{Heintze-Karcher:eq4}, \eqref{Heintze-Karcher:eq5},
    \eqref{Heintze-Karcher:eq6}, and \cite[5.4]{2017arXiv170801549S} to compute
    \begin{align}
        \int_{\Omega \cap C_{\rho}^F} &f(\bm{\psi^F}_{C}(y)) \ud \Leb{n+1}(y)
        = \int_{\Omega \cap \zeta(Z_{\rho})}f(\bm{\psi^F}_{C}(y)) \ud \Leb{n+1}(y) \\
        &= \int_{\Omega \cap \zeta(Z_{\rho})}\int_{\zeta^{-1}(y)}f(z, \mathbf{n}^{F}(C,z)) \ud \Haus{0}(z)\ud \Leb{n+1}(y)\\
        &= \int_{\zeta(Z_{\rho})}\int_{\zeta^{-1}(y)}f(z, \mathbf{n}^{F}(C,z)) \ud \Haus{0}(z)\ud \Leb{n+1}(y)\\
        &= \int_{Z_{\rho}} J_{n+1} \zeta(z,t)\, f(z, \mathbf{n}^{F}(C,z))\ud \Haus{n+1}(z, t) \\
        &= \int_{Q_{\rho}}f(z, \mathbf{n}^{F}(C,z))F(\mathbf{n}(C,z))\int_{0}^{\rho}\Big(1-t\frac{\,H(z)}{n}\Big)^{n}\ud t \ud \Haus{n}(z) \\
        &= \int_{\partial C}f(z, \mathbf{n}^{F}(C,z))F(\mathbf{n}(C,z))\int_{0}^{\rho}\Big(1-t\frac{\,H(z)}{n}\Big)^{n}\ud t \ud \Haus{n}(z)\\
        &= \sum_{i=1}^{n+1}c_{i}(f)\rho^{i},
    \end{align}
    where, for $ i = 1, \ldots , n+1 $,
    \begin{displaymath}
        c_{i}(f) = \Big(-\frac{1}{n}\Big)^{i-1} \frac{n!}{i!(n-i+1)!}
        \int_{\partial C} f(z, \mathbf{n}^{F}(C,z)) F(\mathbf{n}(C,z)) H(z)^{i-1}
        \ud \Haus{n}(z) \,.
    \end{displaymath}
    Therefore, $\reach^{F} C \geq n /c$ by Theorem~\ref{sufficient condition for
      positive reach}.

    \medskip

    Claim 6: \emph{Let $ 0 < r < n/c \leq \reach^{F} C $. Then $S^F(C,r)$ is a
      finite union of Wulff shapes of radii not smaller than $c^{-1}(n-rc)$.}

    \medskip

    Since $\reach^{F} C \ge n/c $ we employ~\ref{cor:level-set-of-class-11} to
    find that $S^F(C,r)$ is a submanifold of~$\Real{n+1}$ of dimension~$n$ of
    class~$\cnt{1,1}$. We define
    \begin{displaymath}
        C_r = \Real{n+1} \cap \{ z : \bm{\delta}^F_C(z) < r \} \,.
    \end{displaymath}
    Noting that $\mathbf{n}^{F}(C_r,\cdot)|S^F(C,r) = \grad F \circ
    \mathbf{n}(C_r,\cdot)|S^F(C,r)$ and $\grad F$ is a~$\cnt{1}$ function,
    we~deduce that $\mathbf{n}^{F}(C_r,\cdot)|S^F(C,r)$ is a~Lipschitzian vector
    field. We define
    \begin{displaymath}
        T = Q \cap \bigl\{ z :  \kappa_{Q, j}^F(z)= -H(z)/n \; \; \text{for $ j = 1, \ldots , n $}  \bigr\} \,,
    \end{displaymath}
    and we notice that $ \Haus{n}(\partial C \without T) = 0$ by
    \eqref{Heintze-Karcher:eq2} and~\eqref{Heintze-Karcher:eq6}; then the
    Lusin~(N) condition implies
    \begin{equation}
        \label{eq:preim-T-SFCr}
        \Haus{n}(S^F(C,r) \without (\bm{\xi}^F_{C})^{-1}(T)) = 0 \,.
    \end{equation}
    Recalling~\ref{lem:xi-diff-ae}\ref{i:xi:direction} we see that
    \begin{gather}
        \mathbf{n}^F(C_r,z)
        = \frac{z - \bm{\xi}^F_C(z)}{r}
        = \grad F \bigl( \mathbf{n}(C,\xi(z)) \bigr)
        = \mathbf{n}^F(C,\cdot) \circ \bm{\xi}^F_C(z)
        \quad \text{whenever $z \in S^F(C,r)$} \,.
    \end{gather}
    Let us set
    \begin{displaymath}
        \sigma = \bm{\xi}^F_C|S^F(C,r)\cap (\bm{\xi}^F_{C})^{-1}(T)
        \quad \text{and} \quad
        \varphi = \zeta|T \times \{r\} \,.
    \end{displaymath}
    Observe that if $x \in T$, then $z = x + r\mathbf{n}^F(C,x) \in S^F(C,r)$,
    $\bm{\xi}^F_C(z) = x$, and $\Tan(S^F(C,r),z) = \Tan(T,x)$; hence, $\sigma =
    \varphi^{-1}$ and we~get
    \begin{gather}
        \label{eq:lip-vaphi}
        \bigl\langle u ,\, \Der \varphi(x) \bigr\rangle
        = (1 - r H(x)/n) u
        \quad \text{for $x \in T$ and $u \in \Tan(T,x)$} \,,
        \\
        \bigl\langle u ,\, \Der \sigma(z) \bigr\rangle
        = (1 - r H(\bm{\xi}^F_C(z))/n)^{-1} u
        \quad \text{for $z \in \dmn \sigma$ and $u \in \Tan(T, \bm{\xi}^F_C(z))$} \,,
        \\
        \Der \mathbf{n}^F(C_r,\cdot)(z)u = \frac{-H(\bm{\xi}^F_C(z))}{n - r H(\bm{\xi}^F_C(z))} u
        \quad \text{for $\Haus{n}$~a.a. $z \in S^F(C,r)$ and $u \in \Tan(T, \bm{\xi}^F_C(z))$} \,.
    \end{gather}
    Employing~\ref{lem:umbilical} we conclude that $S^F(C,r)$ is a~union of at
    most countably many boundaries of Wulff shapes with radii not smaller than
    $c^{-1}(n-rc)$. Since $E$ has finite perimeter we have $\Haus{n}(\partial
    \Omega) < \infty$ so using~\eqref{eq:lip-vaphi} and~\eqref{eq:preim-T-SFCr}
    we conclude that $\Haus{n}(S^F(C,r)) < \Haus{n}(\partial^{\ast} \Omega) <
    \infty $ and Claim~6 follows.

    \medskip

    We are now ready to conclude the proof. We notice from
    \cite[4.20]{Federer1959} that
    \begin{displaymath}
        \partial C = \{ x : \dim \Nor(C,x) \geq 1 \}
    \end{displaymath}
    and by Lemma \ref{normal bundel isotropo vs anisotropo}, we also get that
    \begin{displaymath}
        \partial C = \{ x : \dim \Nor^F(C,x) \geq 1 \} \,.
    \end{displaymath}
    We claim that 
    \begin{equation}
        \bm{\xi}^F_{C}(S^F(C,r)) = \partial C \quad  \text{ for $ 0 < r < n/c $} \,.
    \end{equation}
    Indeed, since $0 < r < \reach^F C$, for every $x \in \partial C$ there
    exists $\nu \in \Nor^F(C,x) $ such that $x + r\nu \in S^F(C,r)\cap \dmn
    \bm{\xi}^F_{C}$ and consequently $\bm{\xi}^F_{C}(x + r\nu)=x$. We deduce
    that $\partial C \subseteq \bm{\xi}^F_{C}(S^F(C,r))$. The~reverse inclusion
    is trivial.

    Consider a connected component~$S_1$ of~$S^F(C,r)$. By Claim~6 we obtain $s
    \ge n/c - r$ and $z \in \Real{n+1}$ such that $S_1=\partial
    \cballF{F^*}{z}{s}$.  Observe that
    \begin{displaymath}
        S^F(\Real{n+1} \without \cballF{F^*}{z}{s+r},r) = S_1 \,;
    \end{displaymath}
    hence,
    \begin{displaymath}
        \partial \cballF{F^*}{z}{s+r} = \bm{\xi}^F_{C}(S_1) \subseteq \partial C 
    \end{displaymath}
    and, using, e.g., the constancy theorem~\cite[4.1.7]{Federer1969}, we deduce
    that $\oballF{F^*}{z}{s+r}$ is a~connected component of~$\Omega$. Since
    $S_1$ was chosen arbitrarily we see that $\Omega$ must be a finite union of
    open disjoint Wulff shapes of radii at least~$n/c$.
\end{proof}

\begin{Remark}
	This theorem extends to sets of finite perimeter the analogous result for smooth boundaries in \cite[Theorem 4]{helimage}.
\end{Remark}

We use now Theorem \ref{Heintze-Karcher} to study the critical points of the anisotropic surface area for a given volume. 

\begin{Definition}[\protect{cf. \cite[4.1]{Allard1972}}]
	A smooth function $ h : (-\epsilon, \epsilon) \times \Real{n+1} \rightarrow \Real{n+1} $ is called \emph{local variation} if and only if
	\begin{enumerate}
	\item $ h(0,x) = x $ for every $ x \in \Real{n+1} $,
	\item $ h(t, \cdot) : \Real{n+1} \rightarrow \Real{n+1} $ is a diffeomorphism for every $ t \in (-\epsilon, \epsilon) $,
	\item the set $ \{x : h(t,x) \neq x \; \textrm{for some $ t \in (-\epsilon, \epsilon) $}   \} $ has compact closure in $ \Real{n+1} $.
	\end{enumerate}

We set $ h_{t} = h(t, \cdot) $ and $ \overset{.}{h}_{t}(x) = \lim_{u \to 0} u^{-1}(h_{t+u}(x) - h_{t}(x)) $ for every $ (t,x) \in (-\epsilon, \epsilon) \times \Real{n+1} $.
\end{Definition}

Given an integrand $ F $ we define \emph{the $ F $-perimeter functional} as 
\begin{equation}\label{F perimeter}
\mathcal{P}_{F}(E) = \int_{\partial^{\ast}E} F(\mathbf{n}(E,x))\, \ud\Haus{n}x
\end{equation}
for every $ E \subseteq \Real{n+1} $ with finite perimeter, and \emph{the $ F $-isoperimetric functional} as
\begin{displaymath}
\mathcal{I}_{F}(E)= \frac{\mathcal{P}_{F}(E)^{n+1}}{\Leb{n+1}(E)^{n}}
\end{displaymath}
for every $E \subseteq \Real{n+1}$ with finite perimeter and finite volume. 

\begin{Corollary}\label{isoperimetric functional}
Let $ E \subseteq \Real{n+1} $ be a set of finite perimeter and finite volume such that 
\begin{equation*}
\Haus{n}(\Clos(\partial^{\ast} E) \without \partial^{\ast}E) = 0.
\end{equation*}
If $ F $ is an elliptic integrand of class $ \mathscr{C}^{3} $ and for every local variation $ h $ it holds that
\begin{equation}\label{var}
\frac{d}{dt}\mathcal{I}_{F}(h_{t}(E))\Big|_{t =0} = 0, 
\end{equation}
then there exists a finite union $ \Omega $ of disjoint open Wulff shapes with equal radii such that 
\begin{equation*}
\Leb{n+1}((\Omega \sim E)\cup(E \sim \Omega)) =0.
\end{equation*}
\end{Corollary}

\begin{proof}
Let $ h $ be a local variation and $ V = \mathbf{v}(\partial^{\ast}E) $. Define $ p(t) = \mathcal{P}_{F}(h_{t}(E)) $ and $ v(t) = \Leb{n+1}(h_{t}(E)) $ for $ -\epsilon < t < \epsilon $. We observe that
\begin{equation*}
a'(0) = \delta_{F}V(\overset{.}{h}_{0}) 
\end{equation*}
\begin{equation*}
	v'(0) = \int_{E} \divergence \overset{.}{h}_{0} \,d\Leb{n+1} = \int_{\partial^{\ast}E} \overset{.}{h}_{0}(x) \bullet \mathbf{n}(E,x)\, \ud\Haus{n}(x)
\end{equation*}
Noting that the derivative in $ t $ the function $ \frac{a^{n+1}}{v^{n}} $ equals
\begin{equation*}
\Big(\frac{p(t)}{v(t)}\Big)^{n}\Big[(n+1)p'(t)  - n \frac{p(t)}{v(t)}v'(t)\Big],
\end{equation*} 
it follows that 
\begin{equation*}
	(n+1)p'(0) - n \frac{p(0)}{v(0)}v'(0) =0
\end{equation*}
and the arbitrariness of $ h $ implies that
\begin{equation*}
\| \delta_{F} V \|_{\sing} =0 \quad \textrm{and} \quad  \overline{\mathbf{h}}_F(V,x) = - \frac{n}{n+1}\frac{\mathcal{P}_{F}(E)}{\Leb{n+1}(E)}\mathbf{n}(E,x).
\end{equation*}
It follows that the hypothesis of Theorem \ref{Heintze-Karcher} and the equality is realized in \eqref{Heintze-Karcher:eq1}. Henceforth, the conclusion follows from Theorem \ref{Heintze-Karcher}.
\end{proof}

\begin{Corollary}\label{main corollary}
Let $ E \subseteq \Real{n+1} $ be a set of finite perimeter and finite volume such that 
  \begin{equation*}
\Haus{n}(\Clos(\partial^{\ast} E) \without \partial^{\ast}E) = 0.
\end{equation*}
If for every local variation $ h $ such that $ \Leb{n+1}(h_{t}(E)) = \Leb{n+1}(E) $ for every $ t \in (-\epsilon, \epsilon) $ it holds that
\begin{equation}\label{var2}
	\frac{d}{dt}\mathcal{P}_{F}(h_{t}(E))\Big|_{t =0} = 0, 
\end{equation}
then there exists a finite union $ \Omega $ of disjoint open Wulff shapes with equal radii such that 
\begin{equation*}
\Leb{n+1}((\Omega \sim E)\cup(E \sim \Omega)) =0.
\end{equation*}
\end{Corollary}

\begin{proof}
Thanks to Corollary \ref{isoperimetric functional}, we just need to prove that such a set $E$ satisfies \eqref{var} for every local variation $ h $. To this aim we define the variation 
$$
f_t(x)=\left(\frac{\Leb{n+1}(E)}{\Leb{n+1}(h_t(E))}\right)^\frac{1}{n+1} h_t(x)
$$
and we observe that for every $ t \in (-\epsilon, \epsilon) $ it holds
$$ \Leb{n+1}(f_{t}(E)) = \left(\frac{\Leb{n+1}(E)}{\Leb{n+1}(h_t(E))}\right)^\frac{n+1}{n+1} \Leb{n+1}(h_t(E))=\Leb{n+1}(E).$$
We deduce from \eqref{var2} that
$$
0=\frac{d}{dt}\mathcal{P}_{F}(f_{t}(E))\Big|_{t =0} = \Leb{n+1}(E)^\frac{n}{n+1}\frac{d}{dt}\frac{\mathcal{P}_{F}(h_{t}(E))}{\Leb{n+1}(h_t(E))^\frac{n}{n+1}}\Big|_{t =0},
$$
which implies \eqref{var}, as desired.
\end{proof}

\subsection*{Acknowledgements}
The first author has been supported by the NSF DMS Grant No.~1906451.
The second author was supported by the National Science Centre Poland
grant no.~2016/23/D/ST1/01084.

\bibliographystyle{plain}

\end{document}